\documentclass[12pt]{amsart}
\usepackage{amsmath}
\usepackage{amssymb}
\usepackage{amsfonts}

\newtheorem{theorem}{Theorem}
\theoremstyle{plain}

\newtheorem{corollary}{Corollary}

\newtheorem{definition}{Definition}
\newtheorem{example}{Example}

\newtheorem{lemma}{Lemma}

\newtheorem{remark}{Remark}

\numberwithin{equation}{section}
\input{tcilatex}

\begin{document}
\title{On the duals of normed spaces and quotient shapes}
\author{Nikica Ugle\v{s}i\'{c}}
\address{Sv. Ante 9, 23287 Veli R\aa t, Hrvatska (Croatia)}
\email{nuglesic@unizd.hr.}
\date{October 26, 2018}
\thanks{This paper is in final form and no version of it will be submitted
for publication elsewhere.}

\begin{abstract}
Some properties of the (normed) dual Hom-functor $D$ and its iterations $%
D^{n}$ are exhibited. For instance: $D$ turns every canonical embedding (in
the second dual space) into a retraction (of the third dual onto the first
one); $D$ rises the countably infinite (algebraic) dimension \emph{only}; $D$
does not change the finite quotient shape type. By means of that, the \emph{%
finite} quotient shape classification of normed vectorial spaces is
completely solved. As a consequence, two extension type theorems are derived.
\end{abstract}

\maketitle

\section{Introduction}

The quotient shape theory is a genuine kind of the general shape theory
(which began as a generalization of the homotopy theory such that the
locally bad spaces can be also considered and classified in a very suitable
\textquotedblleft homotopical\textquotedblright\ way; [1]. [2], [3], [5]
and, especially, [11]). Although, in general, founded purely categorically,
it is mostly well known only as \emph{the} (standard) shape theory of
topological spaces with respect to spaces having the homotopy types of
polyhedra. The generalizations founded in [8] and [18] are, primarily, also
on that line.

The quotient shape theory was introduced a few years ago by the author,
[13]. Though it is a kind of the general (abstract) shape theory, it can be
straightforwardly applied to any concrete category $\mathcal{C}$., whenever
an infinite cardinal $\kappa \geq \aleph _{0}$ is chosen. Concerning a shape
of objects, in general, one has to decide which ones are \textquotedblleft
nice\textquotedblright\ absolutely and/or relatively (with respect to the
chosen ones). In this approach, an object is \textquotedblleft
nice\textquotedblright\ if it is isomorphic to a quotient object belonging
to a special full subcategory and if it (its \textquotedblleft
basis\textquotedblright ) has cardinality less than (less than or equal to)
a given infinite cardinal.\emph{\ }It leads to the basic idea: to
approximate a $\mathcal{C}$-object $X$ by a suitable inverse system
consisting of its quotient objects $X_{\lambda }$ (and the quotient
morphisms) which have cardinalities, or dimensions - in the case of
vectorial spaces, less than (less than or equal to) $\kappa $. Such an
approximation exists in the form of any $\kappa ^{-}$-expansion ($\kappa $%
-expansion) of $X$,

$\boldsymbol{p}_{\kappa ^{-}}=(p_{\lambda }):X\rightarrow \boldsymbol{X}%
_{\kappa ^{-}}=(X_{\lambda },p_{\lambda \lambda ^{\prime }},\Lambda _{\kappa
^{-}})$

($\boldsymbol{p}_{\kappa }=(p_{\lambda }):X\rightarrow \boldsymbol{X}%
_{\kappa }=(X_{\lambda },p_{\lambda \lambda ^{\prime }},\Lambda _{\kappa })$%
),

\noindent where $\boldsymbol{X}_{\kappa ^{-}}$ ($\boldsymbol{X}_{\kappa }$)
belongs to the subcategory $pro$-$\mathcal{D}_{\kappa ^{-}}$ ($pro$-$%
\mathcal{D}_{\kappa }$) of $pro$-$\mathcal{D}$, and $\mathcal{D}_{\kappa
^{-}}$ ($\mathcal{D}_{\kappa }$) is the subcategory of $\mathcal{D}$
determined by all the objects having cardinalities, or dimensions - for
vectorial spaces, less than (less than or equal to) $\kappa $, while $%
\mathcal{D}$ is a full subcategory of $\mathcal{C}$. Clearly, if $X\in Ob(%
\mathcal{D})$ and the cardinality $\left\vert X\right\vert <\kappa $ ($%
\left\vert X\right\vert \leq \kappa $) (or of the \textquotedblleft
baisis\textquotedblright\ of $X$), then the rudimentary pro-morphism $%
\left\lfloor 1_{X}\right\rfloor :X\rightarrow \left\lfloor X\right\rfloor $
is a $\kappa ^{-}$-expansion ($\kappa $-expansion) of $X$. The corresponding
shape category $Sh_{\mathcal{D}_{\kappa ^{-}}}(\mathcal{C})$ ($Sh_{\mathcal{D%
}_{\kappa }}(\mathcal{C})$) and shape functor $S_{\kappa ^{-}}:\mathcal{C}%
\rightarrow Sh_{\mathcal{D}_{\kappa ^{-}}}(\mathcal{C})$ ($S_{\kappa }:%
\mathcal{C}\rightarrow Sh_{\mathcal{D}_{\kappa }}(\mathcal{C})$) exist by
the general (abstract) shape theory, and they have all the appropriate
general properties. Moreover, there exist the relating functors $S_{\kappa
^{-}\kappa }:Sh_{\mathcal{D}_{\kappa }}(\mathcal{C})\rightarrow Sh_{\mathcal{%
D}_{\kappa ^{-}}}(\mathcal{C})$ and $S_{\kappa \kappa ^{\prime }}:Sh_{%
\mathcal{D}_{\kappa ^{\prime }}}(\mathcal{C})\rightarrow Sh_{\mathcal{D}%
_{\kappa }}(\mathcal{C})$, $\kappa \leq \kappa ^{\prime }$, such that $%
S_{\kappa ^{-}\kappa }S_{\kappa }=S_{\kappa ^{-}}$ and $S_{\kappa \kappa
^{\prime }}S_{\kappa ^{\prime }}=S_{\kappa }$. Even in simplest case of $%
\mathcal{D}=\mathcal{C}$, the quotient shape classifications are very often
non-trivial and very interesting. In such a case we simplify the notation $%
Sh_{\mathcal{D}_{\kappa ^{-}}}(\mathcal{C})$ ($Sh_{\mathcal{D}_{\kappa }}(%
\mathcal{C})$) to $Sh_{\kappa ^{-}}(\mathcal{C})$ ($Sh_{\kappa }(\mathcal{C}%
) $) or to $Sh_{\kappa ^{-}}$ ($Sh_{\kappa }$) when $\mathcal{C}$ is fixed.

In [13], several well known concrete categories were considered and many
examples are given which show that the quotient shape theory yields
classifications strictly coarser than those by isomorphisms. In [14] and
[15] were considered the quotient shapes of (purely algebraic, topological
and normed - the category $\mathcal{N}_{F}$) vectorial spaces and
topological spaces, respectively. In paper [16], we have continued the
studying of quotient shapes of normed vectorial spaces of [14], Section 4.1,
primarily and \emph{separately} focused to the well known $l_{p}$ and $L_{p}$
spaces. The main general result of [16] is that the \emph{finite} quotient
shape type of a normed spaces (over the field $F\in \{\mathbb{R},\mathbb{C}%
\} $) reduces to that of its completion (Banach) spaces, and consequently,
that the quotient shape theory of $(\mathcal{N}_{F},(\mathcal{N}_{F})_{\text{%
\b{0}}})$ reduces to that of the full subcategory pair $(\mathcal{B}_{F},(%
\mathcal{B}_{F})_{\text{\b{0}}})$ of Banach spaces. In the very recent paper
[17] we have proven that the finite quotient shape type of normed spaces is
an invariant of the (algebraic) dimension, but not conversely. The
counterexamples exist, at least, in the dimensional par $\{\aleph
_{0},2^{\aleph _{0}}\}$. Further, in the case of separable Banach spaces,
the classifications by dimension and by the finite quotient shape (as well
as by the countable quotient shape) coincide. An application of those
results has yielded two\ extension type theorems into lower dimensional
Banach spaces.

In this paper we firstly consider the (iterated) dual normed spaces. We have
used the functorial approach, i.e., we treat $D^{2n-1}:\mathcal{N}%
_{F}\rightarrow \mathcal{B}_{F}$ ($D^{2n}:\mathcal{N}_{F}\rightarrow 
\mathcal{B}_{F}$), $n\in \mathbb{N}$, as a contravariant (covariant) $%
Hom_{F} $-functor. Among interesting and useful results, let us quote that,
for every $X$, $D^{2n-1}$ turns the canonical embedding $j:X\rightarrow
D^{2}(X)$ into $D^{2n-1}(j):D^{2n+1}(X)\rightarrow D^{2n-1}(X)$ which is a
retraction, while $D^{2n}(j):D^{2n}(X)\rightarrow D^{2n+2}(X)$ is a section,
and thus, one may consider $D^{n}(X)$ to be a retract of $D^{n+2}(X)$
admitting a closed direct complement. Therefore, the problem of the strict
relation between $\dim D(X)$ and $\dim X$ (generally, $\dim X\leq \dim D(X)$%
) occurs in this consideration as a very significant one. We have solved it
by means of the quotient shape theory technique as follows:

$\dim D(X)>\dim X$\ if and only if $\dim X=\aleph _{0}$.

\noindent By applying that fact, we made an important step towards the main
goal which was the complete \emph{finite} quotient shape classification of
all normed vectorial spaces over $F\in \{\mathbb{R},\mathbb{C}\}$. Briefly,
for all but countably infinite-dimensional normed spaces, the finite
quotient shape types are their (algebraic) dimension classes, while all
countably infinite-dimensional normed spaces belong to the dimension class $%
2^{\aleph _{0}}$. Consequently, every finite quotient shape type of normed
spaces contains a representative that is a Hilbert space, and in addition,
the finite quotient shape type of all $\aleph _{0}$- and all $2^{\aleph
_{0}} $-dimensional normed spaces admits an $\aleph _{0}$-dimensional
unitary representative.

At the end, we have established a significant improvements of the mentioned
(previously obtained, [17]) extension type theorems. For instance (Theorem
6):

\emph{Let }$X$\emph{\ be a normed vectorial space, let }$Y$\emph{\ be a
Banach space (over the same field) and let }$f_{n}:D^{n}(X)\rightarrow Y$%
\emph{, }$n\in N$\emph{, be a continuous linear function. Then, for every }$%
k\in \{0\}\cup N$\emph{, }$f_{n}$\emph{\ admits a continuous linear
norm-preserving extension }$f_{n,k}:D^{n+2k}(X)\rightarrow Y$.

\section{Preliminaries}

We shall frequently use and apply in the sequel several general or special
well known facts without referring to any source. So we remind a reader that

\noindent - our general shape theory technique is that of [11];

\noindent - the needed set theoretic (especially, concerning cardinals) and
topological facts can be found in [4];

\noindent - the facts concerning functional analysis are taken from [6],
[9], [10] or [12];

\noindent - our category theory language follows that of [7].

For the sake of completeness, let us briefly repeat the construction of a
quotient shape category and a quotient shape functor, [13]. Given a category
pair $(\mathcal{C},\mathcal{D}$), where $\mathcal{D}\subseteq \mathcal{C}$
is full, and a cardinal $\kappa $, let $\mathcal{D}_{\kappa ^{-}}$ ($%
\mathcal{D}_{\kappa }$) denote the full subcategory of $\mathcal{D}$
determined by all the objects having cardinalities or, in some special
cases, the cardinalities of \textquotedblleft bases\textquotedblright\ less
than (less or equal to) $\kappa $. By following the main principle, let $(%
\mathcal{C},\mathcal{D}_{\kappa ^{-}})$ ($(\mathcal{C},\mathcal{D}_{\kappa
}) $) be such a pair of \emph{concrete} categories. If

\noindent (a) every $\mathcal{C}$-object $(X,\sigma )$ admits a directed set 
$R(X,\sigma ,\kappa ^{-})\equiv \Lambda _{\kappa ^{-}}$ ($R(X,\sigma ,\kappa
)\equiv \Lambda _{\kappa }$) of equivalence relations $\lambda $ on $X$ such
that each quotient object $(X/\lambda ,\sigma _{\lambda })$ has to belong to 
$\mathcal{D}_{\kappa ^{-}}$ ($\mathcal{D}_{\kappa }$), while each quotient
morphism $p_{\lambda }:(X,\sigma )\rightarrow (X/\lambda ,\sigma _{\lambda
}) $ has to belong to $\mathcal{C}$;

\noindent (b) the induced morphisms between quotient objects belong to $%
\mathcal{D}_{\kappa ^{-}}$ ($\mathcal{D}_{\kappa }$);

\noindent (c) every morphism $f:(X,\sigma )\rightarrow (Y,\tau )$ of $%
\mathcal{C}$, having the codomain in $\mathcal{D}_{\kappa ^{-}}$ ($\mathcal{D%
}_{\kappa }$), factorizes uniquely through a quotient morphism $p_{\lambda
}:(X,\sigma )\rightarrow (X/\lambda ,\sigma _{\lambda })$, $f=gp_{\lambda }$%
, with $g$ belonging to $\mathcal{D}_{\kappa ^{-}}$ (\thinspace $\mathcal{D}%
_{\kappa }$),

\noindent then $\mathcal{D}_{\kappa ^{-}}$ ($\mathcal{D}_{\kappa }$) is a
pro-reflective subcategory of $\mathcal{C}$. Consequently, there exists a
(non-trivial) \emph{(quotient)\ shape\ category }$Sh_{(\mathcal{C},\mathcal{D%
}_{\kappa ^{-}})}\equiv Sh_{\mathcal{D}_{\kappa ^{-}}}(\mathcal{C})$ ($Sh_{(%
\mathcal{C},\mathcal{D}_{\kappa })}\equiv Sh_{\mathcal{D}_{\kappa }}(%
\mathcal{C})$) obtained by the general construction.

Therefore, a $\kappa ^{-}$-shape morphism $F_{\kappa ^{-}}:(X,\sigma
)\rightarrow (Y,\tau )$ is represented by a diagram (in $pro$-$\mathcal{C}$)

$%
\begin{array}{ccc}
(\boldsymbol{X},\boldsymbol{\sigma })_{\kappa ^{-}} & \overset{p_{\kappa
^{-}}}{\leftarrow } & (X,\sigma ) \\ 
\boldsymbol{f}_{\kappa ^{-}}\downarrow &  &  \\ 
(\boldsymbol{Y},\boldsymbol{\tau })_{\kappa ^{-}} & \overset{q_{\kappa ^{-}}}%
{\leftarrow } & (Y,\tau )%
\end{array}%
$

\noindent (with $\boldsymbol{p}_{\kappa ^{-}}$ and $\boldsymbol{q}_{\kappa
^{-}}$ - a pair of appropriate expansions), and similarly for a $\kappa $%
-shape morphism $F_{\kappa }:(X,\sigma )\rightarrow (Y,\tau )$. Since all $%
\mathcal{D}_{\kappa ^{-}}$-expansions ($\mathcal{D}_{\kappa }$-expansions)
of a $\mathcal{C}$-object are mutually isomorphic objects of $pro$-$\mathcal{%
D}_{\kappa ^{-}}$ ($pro$-$\mathcal{D}_{\kappa }$), the composition and
identities follow straightforwardly. Observe that every quotient morphism $%
p_{\lambda }$ is an effective epimorphism. (If $U$ is the forgetful functor,
then $U(p_{\lambda })$ is a surjection), and thus condition (E2) for an
expansion follows trivially.

The corresponding \emph{\textquotedblleft quotient\ shape\textquotedblright\
functors} $S_{\kappa ^{-}}:\mathcal{C}\rightarrow Sh_{\mathcal{D}_{\kappa
^{-}}}(\mathcal{C})$ and $S_{\kappa }:\mathcal{C}\rightarrow Sh_{\mathcal{D}%
_{\kappa }}(\mathcal{C})$ are defined in the same general manner. That means,

$S_{\kappa ^{-}}(X,\sigma )=S_{\kappa }(X,\sigma )=(X,\sigma )$;

if $f:(X,\sigma )\rightarrow (Y,\tau )$ is a $\mathcal{C}$-morphism, then,
for every $\mu \in M_{\kappa ^{-}}$, the composite $g_{\mu }f:(Y,\tau
)\rightarrow (Y_{\mu },\tau _{\mu })$ factorizes (uniquely) through a $%
p_{\lambda (\mu )}:(X,\sigma )\rightarrow (X_{\lambda (\mu )},\sigma
_{\lambda (\mu )}),$ and thus, the correspondence $\mu \mapsto \lambda (\mu
) $ yields a function $\phi :M_{\kappa ^{-}}\rightarrow \Lambda _{\kappa
^{-}}$ and a family of $\mathcal{D}_{\kappa ^{-}}$-morphisms $f_{\mu
}:(X_{\phi (\mu )},\sigma _{\phi (\mu )})\rightarrow (Y_{\mu },\tau _{\mu })$
such that $q_{\mu }f=f_{\mu }p_{\phi (\mu )}$;

\noindent one easily shows that $(\phi ,f_{\mu }):(\boldsymbol{X},%
\boldsymbol{\sigma })_{\kappa ^{-}}\rightarrow (\boldsymbol{Y},\boldsymbol{%
\tau })_{\kappa ^{-}}$ is a morphism of $inv$-$\mathcal{D}_{\kappa ^{-}}$,
so the equivalence class $\boldsymbol{f}_{\kappa ^{-}}=[(\phi ,f_{\mu })]:(%
\boldsymbol{X},\boldsymbol{\sigma })_{\kappa ^{-}}\rightarrow (\boldsymbol{Y}%
,\boldsymbol{\tau })_{\kappa ^{-}}$ is a morphism of $pro$-$\mathcal{D}%
_{\kappa ^{-}}$;

\noindent then we put $S_{\kappa ^{-}}(f)=\left\langle \boldsymbol{f}%
_{\kappa ^{-}}\right\rangle \equiv F_{\kappa ^{-}}:(X,\sigma )\rightarrow
(Y,\tau )$ in $Sh_{\mathcal{D}_{\kappa ^{-}}}(\mathcal{C})$.

\noindent The identities and composition are obviously preserved. In the
same way one defines the functor $S_{\kappa }$.

Furthermore, since $(\boldsymbol{X,\sigma })_{\kappa ^{-}}$ is a subsystem
of $(\boldsymbol{X,\sigma })_{\kappa }$ (more precisely, $(\boldsymbol{%
X,\sigma })_{\kappa }$ is a subobject of $(\boldsymbol{X,\sigma })_{\kappa
^{-}}$ in $pro$-$\mathcal{D}$), one easily shows that there exists a functor 
$S_{\kappa ^{-}\kappa }:Sh_{\mathcal{D}_{\kappa }}(\mathcal{C})\rightarrow
Sh_{\mathcal{D}_{\kappa ^{-}}}(\mathcal{C})$ such that $S_{\kappa ^{-}\kappa
}S_{\kappa }=S_{\kappa ^{-}}$, i.e., the diagram

$%
\begin{array}{ccccc}
&  & \mathcal{C} &  &  \\ 
& \swarrow S_{\kappa ^{-}} &  & S_{\kappa }\searrow &  \\ 
Sh_{\mathcal{D}_{\kappa ^{-}}}(\mathcal{C}) &  & \underleftarrow{S_{\kappa
^{-}\kappa }} &  & Sh_{\mathcal{D}_{\kappa }}(\mathcal{C})%
\end{array}%
$

\noindent commutes. Moreover, an analogous functor $S_{\kappa \kappa
^{\prime }}:Sh_{\mathcal{D}_{\kappa ^{\prime }}}(\mathcal{C})\rightarrow Sh_{%
\mathcal{D}_{\kappa }}(\mathcal{C})$, satisfying $S_{\kappa \kappa ^{\prime
}}S_{\kappa ^{\prime }}=S_{\kappa }$, exists for every pair of infinite
cardinals $\kappa \leq \kappa ^{\prime .}$

Generally, in the case of $\kappa =\aleph _{0}$, the $\kappa ^{-}$-shape is
said to be the \emph{finite (quotient) shape}, because all the objects in
the expansions are of finite (bases) cardinalities, and the category is
denoted by $Sh_{\mathcal{D}_{\underline{0}}}(\mathcal{C})$ or by $Sh_{\text{%
\b{0}}}(\mathcal{C})\equiv Sh_{\underline{0}}$ only, whenever $\mathcal{D}=%
\mathcal{C}$.

Let us finally notice that, though $\mathcal{D}\nsubseteq \mathcal{C}%
_{\kappa ^{-}})$ ($(\mathcal{D\nsubseteq C}_{\kappa })$), the quotient shape
category $Sh_{\mathcal{C}_{\kappa ^{-}}}(\mathcal{D})$ ($Sh_{\mathcal{C}%
_{\kappa }}(\mathcal{D})$) exists as a full subcategory of $Sh_{\mathcal{C}%
_{\kappa ^{-}}}(\mathcal{C})$ ($Sh_{\mathcal{C}_{\kappa }}(\mathcal{C})$),
and, if $\mathcal{D}$ is closed with respect to quotients, then $Sh_{%
\mathcal{C}_{\kappa ^{-}}}(\mathcal{D})=Sh_{\mathcal{D}_{\kappa ^{-}}}(%
\mathcal{D})$ ($Sh_{\mathcal{C}_{\kappa }}(\mathcal{D})=Sh_{\mathcal{D}%
_{\kappa }}(\mathcal{D})$).

\section{Some properties of the (normed) dual functors}

We recall hereby the well known dual space of a normed vectorial space over $%
F\in \{\mathbb{R},\mathbb{C}\}$. For our purpose it is much more convenient
to use the categorical approach as follows. There exists a contravariant
structure preserving hom-functor, i.e., the contravariant Hom-functor

$Hom_{F}\equiv D:\mathcal{N}_{F}\rightarrow \mathcal{N}_{F}$,

$D(X)=X^{\ast }$ - the (normed) dual space of $X$,

$D(f:X\rightarrow Y)\equiv D(f)\equiv f^{\ast }:Y^{\ast }\rightarrow X^{\ast
}$, $D(f)(y^{1})=y^{1}f$,

\noindent and $D[\mathcal{N}_{F}]\subseteq \mathcal{B}_{F}$. Furthermore,
for every ordered pair $X,Y\in Ob(\mathcal{N}_{F})$, the function

$D_{Y}^{X}:\mathcal{N}_{F}(X,Y)\equiv L(X,Y)\rightarrow L(Y^{\ast },X^{\ast
})\equiv \mathcal{B}_{F}(Y^{\ast },X^{\ast })$

\noindent is a linear isometry ($\left\Vert D(f)\right\Vert =\left\Vert
f\right\Vert $), and hence, $D_{Y}^{X}$ belongs to $Mor(\mathcal{N}_{F})$
and $D$ is a faithful functor.

Further, there exists a covariant Hom-functor

$Hom_{F}^{2}\equiv D^{2}:\mathcal{N}_{F}\rightarrow \mathcal{N}_{F}$,

$D^{2}(X)=D(D(X))\equiv X^{\ast \ast }$ - the (normed) second dual space of $%
X$,

$D^{2}(f:X\rightarrow Y)\equiv D^{2}(f)=D(D(f))\equiv f^{\ast \ast }:X^{\ast
\ast }\rightarrow Y^{\ast \ast }$, $D^{2}(f)(x^{2})=x^{2}D(f)$,

\noindent and $D^{2}[\mathcal{N}_{F}]\subseteq \mathcal{B}_{F}$. (Caution:
The notation \textquotedblleft $D(D(f)(x^{2}))$\textquotedblright\ makes no
sense!) Furthermore, for every ordered pair $X,Y\in Ob(\mathcal{N}_{F})$,
the function

\noindent $(D^{2})_{Y}^{X}:\mathcal{N}_{F}(X,Y)\equiv L(X,Y)\rightarrow
L(X^{\ast \ast },Y^{\ast \ast })\equiv \mathcal{B}_{F}(X^{\ast \ast
},Y^{\ast \ast })$

\noindent is a linear isometry ($\left\Vert D^{2}(f)\right\Vert =\left\Vert
f\right\Vert $), and thus, $(D^{2})_{Y}^{X}$ belongs to $Mor(\mathcal{N}%
_{F}) $ and $D^{2}$ is a faithful functor.

. The most useful fact hereby is the existence of a certain natural
transformation $j:1_{\mathcal{N}_{F}}\rightsquigarrow D^{2}$ of the
functors, where, for every $X$, $j_{X}:X\rightarrow D^{2}(X)$ is an
isometric embedding (the \emph{canonical} embedding defined by $%
(j_{X}(x))(x^{1})=x^{1}(x)$), and $Cl(j_{X}[X])\subseteq D^{2}(X)$ is the
well known (Banach) completion of $X$. Namely, given a pair $X$, $Y$ of
normed spaces, then

$(\forall f\in \mathcal{N}_{F}(X,Y)$, $j_{Y}f=D^{2}(f)j_{X}$

\noindent holds true. Indeed, for every $x\in X$, every $y^{1}\in D^{1}(Y)$
and every $x^{2}\in D^{2}(X)$,

($(j_{Y}f)(x))(y^{1})=y^{1}(f(x))=j_{X}(x)(y^{1}f)=j_{X}(x)(D(f)(y^{1}))=$

$%
(j_{X}(x)D(f))(y^{1})=(D^{2}(f)(j_{X}(x))(y^{1})=((D^{2}(f)j_{X})(x))(y^{1}) 
$.

Clearly, if $X$ is a Banach space, then the canonical embedding $j_{X}$ is
closed. Continuing by induction, for every $n\in \mathbb{N}$, $n>2$, there
exists a $Hom_{F}$-functor $D^{n}$ of $\mathcal{N}_{F}$ to $\mathcal{N}_{F}$
such that $D^{n}[\mathcal{N}_{F}]\subseteq \mathcal{B}_{F}$, $D^{n}$ is
contravariant (covariant) whenever $n$ is odd (even), and for every ordered
pair $X$, $Y$ of normed spaces, the function $(D^{n})_{Y}^{X}$ is an
isometric linear morphism of the normed space $L(X,Y)$ to the Banach space $%
L(D^{n}(Y),D^{n}(X))$ ($n$ odd) or $L(D^{n}(X),D^{n}(Y))$ ($n$ even).
Consequently, every $(D^{n})_{Y}^{X}$ preserves null-morphisms, i.e., $%
D^{n}(c_{\theta })=c_{\theta }^{n}$. However, it holds much more than that.

\begin{lemma}
\label{L1}(i) The functor $D$ turns

\noindent - (open) epimorphisms into (closed) monomorphisms;

\noindent - open or closed monomorphisms and embeddings into (open)
epimorphisms;

\noindent - isometric isomorphisms into isometric isomorphisms.

\noindent The functor $D^{2}$ maps

\noindent - open epimorphisms into (open) epimorphisms;

\noindent - open or closed monomorphisms and embeddings into closed
monomorphisms;

\noindent - isometries into closed isometries.

\noindent (ii) In addition, the restriction functor $D$\TEXTsymbol{\vert}$%
\mathcal{B}_{F}$ turns

\noindent - epimorphisms into closed monomorphisms;

\noindent - (isometric) monomorphisms with closed ranges into (closed)
epimorphisms.

\noindent The restriction functor $D^{2}|\mathcal{B}_{F}$ maps

\noindent - epimorphisms into epimorphisms;

\noindent - monomorphisms with closed ranges into closed monomorphisms.

\noindent (iii) For all $X,Y\in Ob\mathcal{N}_{F}$, the canonical embedding

$j_{L(X,Y)}:L(X,Y)\rightarrow D^{2}(L(X,Y))$

\noindent factorizes trough the linear isometry

$(D^{2})_{Y}^{X}:L(X,Y)\rightarrow L(D^{2}(X),D^{2}(Y))$, $%
D^{2}(f)(x^{2})=x^{2}D(f)$.

\noindent If $Y$ is a Banach space, then the linear isometry

$D_{Y}^{X}:L(X,Y)\rightarrow L(D(Y),D(X))$, $D(f)(y^{1})=y^{1}f$,

\noindent is closed.
\end{lemma}

\begin{proof}
(i). Assume that $f\in \mathcal{N}_{F}(X,Y)$ is an epimorphism. Let $%
y^{1},y^{1\prime }\in Y^{\ast }$ such that $D(f)(y^{1})=D(f)(y^{1\prime })$.
It means that $y^{1}f=y^{1\prime }f$, implying that $y^{1}=y^{1\prime }$
because $f$ is an epimorphism. Hence, $D(f)$ is a monomorphism of the
underlying abelian groups, and consequently, it is a monomorphism of $%
\mathcal{B}_{F}\subseteq \mathcal{N}_{F}$. Assume, in addition, that $f$ is
open. It suffices to prove that the range $R(D(f))\trianglelefteq X^{\ast }$
is a closed subspace, i.e., that $Cl(R(D(f)))\subseteq R(D(f))$. Namely, if
it is so, then $D(f)$ is a monomorphism of a Banach space with the range
that is a Banach space too. Then,

$D(f)^{\prime }:Y^{\ast }\rightarrow R(D(f))$, $D(f)^{\prime
}(y^{1})=D(f)(y^{1})$,

\noindent is a continuous bijection of Banach spaces, and thus, an
isomorphism, implying that $D(f)$ is closed monomorphism. Let $x^{1}\in
Cl(R(D(f))).$ Consider a sequence $(x_{n}^{1})$ in $R(D(f))$ such that $\lim
(x_{n}^{1})=x^{1}$. Since $D(f)$ is a monomorphism, there exists a unique
sequence $(y_{n}^{1})$ in $Y^{\ast }$ such that, for each $n\in \mathbb{N}$,

$D(f)(y_{n}^{1})=y_{n}^{1}f=x_{n}^{1}$.

\noindent Recall that, algebraically, $X=N(f)\overset{\cdot }{+}W$, where $%
W\cong R(f)=Y$, and that each fiber $f^{-1}[\{y\}]$, $y\in Y$, is the
equivalence class $[x]_{f}=x+N(f)$, where $f(x)=y$. Thus, for every $n$, and
every $y\in Y$,

$y_{n}^{1}(y)=x_{n}^{1}(x)=x_{n}^{1}(w)$,

\noindent where $f(x)=y$ and $x=z+w$ is the unique presentation of $x\in
X=N(f)\overset{\cdot }{+}W$. It implies that, for each $y\in Y$ and all $%
x=z+w\in X$, such that $f(x)=y$,

$\lim (y_{n}^{1}(y))=\lim (x_{n}^{1}(x))=x^{1}(x)=x^{1}(w)$

\noindent holds true. Consequently, by putting

$y^{1}:Y\rightarrow F$, $y^{1}(y)=\lim (y_{n}^{1}(y))$,

\noindent a certain function is well defined. Moreover, $y^{1}$ is linear,
because it is a \textquotedblleft copy\textquotedblright\ of the restriction 
$x^{1}|W$, and $y^{1}f=x^{1}$ obviously holds. It remains to prove that $%
y^{1}$ is continuous. Let $O$ be an open neighborhood of $0\in F$. Since $%
x^{1}$ is continuous, there exists an open neighborhood $U$ of $\theta
_{X}\in X$ such that $x^{1}[U]\subseteq O$. Then $V\equiv f[U]$ is an open
neighborhood of $\theta _{Y}\in Y$, because $f$ is open, and

$y^{1}[V]=(y^{1}f)[U]=x^{1}[U]\subseteq O$.

\noindent Thus, $y^{1}$ is continuous, and hence $y^{1}\in Y^{\ast }$. Since 
$D(f)(y^{1})=y^{1}f=x^{1}$, the additional statement is proven.

\noindent Assume that $f:X\rightarrow Y$ is an open or closed monomorphism
or an embedding. Then $f$ admits the factorization

$%
\begin{array}{ccccc}
X & \overset{f^{\prime }}{\rightarrow } & f[X] & \overset{i}{\hookrightarrow 
} & Y%
\end{array}%
$, $f^{\prime }(x)=f(x)$,

\noindent where $f^{\prime }$ is an isomorphism onto the subspace $%
f[X]\trianglelefteq Y$, and $i$ is the inclusion. Given an $x^{1}\in X^{\ast
}$, put $y_{x^{1}}^{1}=x^{1}f^{\prime -1}\in f[X]^{\ast }$. By the
Hahn-Banach theorem, there exists an extension $y^{1}\in Y^{\ast }$ of $%
y_{x^{1}}^{1}$, i.e., $y^{1}i=y_{x^{1}}^{1}$. Then

$D(f)(y^{1})=D(if^{\prime })(y^{1})=(D(f^{\prime })D(i))(y^{1})=D(f^{\prime
})(D(i)(y^{1}))=$

$=D(f^{\prime })(y^{1}i)=D(f^{\prime
})(y_{x^{1}}^{1})=y_{x^{1}}^{1}f^{\prime }=x^{1}f^{\prime -1}f^{\prime
}=x^{1}$,

\noindent implying that $D(f):Y^{\ast }\rightarrow X^{\ast }$ is an
epimorphism of the underlying abelian groups, and consequently, it is an
epimorphism of $\mathcal{B}_{F}\subseteq \mathcal{N}_{F}$. Now, by
Open-mapping theorem, $D(f)$ is open. Finally, if $f$ is an isometric
isomorphism, then $D(f)$ is an isomorphism and, for every $y^{1}\in D(Y)$,

$\left\Vert D(f)(y^{1})\right\Vert =\left\Vert y^{1}f\right\Vert =\sup
\{\left\vert y^{1}(f(x))\right\vert \mid x\in X,\left\Vert x\right\Vert
=1\}= $

$=\sup \{\left\vert y^{1}(y)\right\vert \mid y\in Y,\left\Vert y\right\Vert
=\left\Vert f(x)\right\Vert =\left\Vert x\right\Vert =1\}=\left\Vert
y^{1}\right\Vert $.

\noindent Hence, $D(f)$ is an isometry as well. The statements concerning $%
D^{2}$ follow by $D^{2}(f)=D(D(f))$ and $D^{2}(f)j_{X}=j_{Y}f$, where $j_{X}$
and $j_{Y}$ are the (isometric) canonical embeddings.

\noindent (ii). Assume that $f\in \mathcal{B}_{F}(X,Y)$ is an epimorphism.
By Open-mapping theorem, $f$ is open as well. Then, by (i), $D(f)$ is a
closed monomorphism.

\noindent Assume that $f\in \mathcal{B}_{F}(X,Y)$ is a monomorphism having
the range $R(f)$ closed in $Y$. Then, as previously,

$f^{\prime }:X\rightarrow R(f)$, $f^{\prime }(x)=f(x)$,

\noindent is a continuous bijection of Banach spaces, and thus, an
isomorphism. It follows that $f$ is a closed monomorphism. Then, by (i), $%
D(f)$ is an (open) epimorphism. If, in addition, $f$ is an isometry, then $f$
preserves Cauchy sequences, and one readily verifies that $D(f)$ maps the
sets closed in $D(Y)$ into sets closed in $D(X)$. The statements concerning $%
D^{2}|\mathcal{B}_{F}$ follow by those concerning $D|\mathcal{B}_{F}$.

\noindent (iii). Consider the range

$R((D^{2})_{Y}^{X})\equiv (D^{2})_{Y}^{X}[L((X,Y))]\trianglelefteq
L(D^{2}(X),D^{2}(Y))$

\noindent and the function

$u:R((D^{2})_{Y}^{X})\rightarrow D^{2}(L(X,Y))$

\noindent well defined by $u(D^{2}(f))=f_{f}^{2}$ such that, for each $%
f^{1}\in D(L(X,Y))$, $f_{f}^{2}(f^{1})=f^{1}(f)$. One readily sees that $u$
is linear and continuous, and that $j_{L(X,Y)}=u(D^{2})_{Y}^{X}$.

\noindent Finally, let $Y$ be a Banach space, and let $C\subseteq L(X,Y)$ be
a closed set. Let $(g_{n})$ be sequence in $D[C]$ that converges in $%
L(D(Y),D(X))$, i.e., there exists $\lim (g_{n})\equiv g\in L(D(Y),D(X))$.
Since $D_{Y}^{X}$ is a linear isometry, it is a monomorphism, and there
exists a unique Cauchy sequence $(f_{n})$ in $C$ such that, for each $n\in 
\mathbb{N}$, $D(f_{n})=g_{n}$. Notice that $L(X,Y)$ is a Banach space
because such is $Y$, and thus, there exists $\lim (f_{n})\equiv f\in L(X,Y)$%
. Since $C\subseteq L(X,Y)$ is closed, it follows that $f\in C$. Then $%
D(f)\in D[C]$, and the continuity implies that $D(f)=g$, which completes the
proof.
\end{proof}

It is well known that there are Banach spaces (for instance $l_{1}$ and $%
c_{0}$) that are not \emph{isometrically isomorphic} to any of the dual
normed spaces. We shall now prove that it holds in general, i.e., without
\textquotedblleft isometrically\textquotedblright .

\begin{lemma}
\label{L2}(i) If a normed space $X$ is isomorphic to a dual space of a
normed space, then $X$ is a Banach space, the canonical embedding $%
j_{X}:X\rightarrow D^{2}(X)$ is a section (of $\mathcal{B}_{F}$) and $%
X\equiv R(j_{X})$ admits a closed direct complement in $D^{2}(X)$.

\noindent (ii) The (codomain restriction) functor $D:\mathcal{N}%
_{F}\rightarrow \mathcal{B}_{F}$ is not surjective onto the object class of
any skeleton of $\mathcal{B}_{F}$, i.e.,

$(\exists X\in Ob\mathcal{B}_{F})(\forall Y\in Ob\mathcal{N}_{F})$ $X\ncong
D(Y)$.
\end{lemma}

\begin{proof}
(i). Let $X$ be a normed space such that $X\cong D^{n}(Y)$ for some $Y\in Ob%
\mathcal{N}_{F}$ and $n\in \mathbb{N}$. Since $D^{n}(Y)=D(D^{n-1}(Y))$, one
may assume that $n=1$. Let $f:X\rightarrow D(Y)$ be an isomorphism of $%
\mathcal{N}_{F}$. Then $D^{2}(f):D^{2}(X)\rightarrow D^{3}(Y)$ is an
isomorphism of $\mathcal{B}_{F}\subseteq \mathcal{N}_{F}$ and the diagram

$%
\begin{array}{ccc}
X & \overset{f}{\rightarrow } & D(Y) \\ 
j_{X}\downarrow &  & \downarrow j_{D(Y)} \\ 
D^{2}(X) & \underset{D^{2}(f)}{\rightarrow } & D^{3}(Y)%
\end{array}%
$

\noindent in $\mathcal{N}_{F}$ commutes. We are to prove that $%
j_{D(Y)}:D(Y)\rightarrow D^{3}(Y)$ is a section (of $\mathcal{B}_{F})$
having $D(j_{Y}):D^{3}(Y)\rightarrow D(Y)$ for a corresponding retraction.
Recall that $j_{Y}:Y\rightarrow D^{2}(Y)$ is defined by $%
j_{Y}(y^{0})=y_{y^{0}}^{2}$, $y^{0}\in Y$, such that, for every $y^{1}\in
D(Y)$, $y_{y^{0}}^{2}(y^{1})=y^{1}(y^{0})$. Further,

$D(j_{Y}):D(D^{2}(Y))=D^{3}(Y)\rightarrow D(Y)$

\noindent is determined by $D(j_{Y})(y^{3})=y^{3}j_{Y}$, $y^{3}\in D^{3}(Y)$%
. In the same way, the canonical embedding

$j_{D(Y)}:D(Y)\rightarrow D^{2}(D(Y))=D^{3}(Y)=D(D^{2}(Y))$

\noindent is determined by $j_{D(Y}(y^{1})=y_{y^{1}}^{3}$, $y^{1}\in D(Y)$,
such that, for every $y^{2}\in D^{2}(Y)$, $y_{y^{1}}^{3}(y^{2})=y^{2}(y^{1})$%
. Then, for every $y^{1}\in D(Y)$,

$%
(D(j_{Y})j_{D(Y)})(y^{1})=D(j_{Y})(j_{D(Y)}(y^{1}))=D(j_{Y})(y_{y^{1}}^{3})=y_{y^{1}}^{3}j_{Y} 
$.

\noindent Since, in addition, for every $y^{0}\in Y$,

$%
(y_{y^{1}}^{3}j_{Y})(y^{0})=y_{y^{1}}^{3}(j_{Y}(y^{0}))=y_{y^{1}}^{3}(y_{y^{0}}^{2})=y_{y^{0}}^{2}(y^{1})=y^{1}(y^{0}) 
$,

\noindent it follows that, for every $y^{1}\in D(Y)$, $%
y_{y^{1}}^{3}j_{Y}=y^{1}$ holds, and therefore $D(j_{Y})j_{D(Y)}=1_{D(Y)}$.
This proves the claim. (Notice that $j_{D(Y)}D(j_{Y}):D^{3}(Y)\rightarrow
D^{3}(Y)$ is a projection of norm $1$.) Put

$r_{X}:D^{2}(X)\rightarrow X$, $r_{X}=f^{-1}D(j_{Y})D^{2}(f)$.

\noindent Then one readily verifies that $r_{X}$ is a retraction of $%
\mathcal{N}_{F}$ having $j_{X}$ for a corresponding section, i.e., $%
r_{X}j_{X}=1_{X}$. Consequently, $R(j_{X})\trianglelefteq D^{2}(X)$ is a
retract of $D^{2}(X)$, and thus, a closed subspace, implying that it is a
Banach space. Since the canonical embedding $j_{X}$ is an isometry, it
follows that $X$ is a Banach space as well. Then, clearly, $j_{X}$ and $%
r_{X} $ belong to $\mathcal{B}_{F.}F$Further, notice that the morphism

$p_{X}\equiv j_{X}r_{X}:D^{2}(X)\rightarrow D^{2}(X)$

\noindent is a continuous linear projection ($p_{X}^{2}=p_{X}$) onto $%
R(j_{X})$. Therefore, the Banach space $X$, identified with $j_{X}[X]\equiv
R(j_{X})$, admits a closed direct complement in $D^{2}(X)$. (Notice that $%
\left\Vert p_{X}\right\Vert =\left\Vert r_{X}\right\Vert =1$ regardless to $%
\left\Vert f\right\Vert $.)

\noindent (ii). Assume to the contrary, i.e., that every Banach space is
isomorphic to the dual space of a normed space. Then, since the dual of a
space equals to the dual of its Banach completion, every Banach space would
be isomorphic to the dual of a Banach space. Further, by iteration, every
Banach space would isomorphic to the second dual of a Banach space. Let $X$
be a non-bidual-like Banach space, i.e., $D^{2}(X)\not\cong X$ (see [17],
Lemma 4. (ii)). Consider any closed subspace $Z\trianglelefteq D^{2}(X)$
such that $X\equiv R(j_{X})\trianglelefteq Z\trianglelefteq D^{2}(X)$, and
denote by $i:X\hookrightarrow Z$ the inclusion. Then we may assume that $%
D^{2}(X)\trianglelefteq D^{2}(Z)$ as well. By (i), there exists a retraction 
$r_{Z}$ corresponding to the canonical embedding $j_{Z}$, i.e.,

$r_{Z}:D^{2}(Z)\rightarrow Z$, $r_{Z}j_{Z}=1_{Z}$.

\noindent Then the domain restriction

$r\equiv r_{Z}|D^{2}(X):D^{2}(X)\rightarrow Z$, $ri=1_{Z}$,

\noindent is a (continuous linear) retraction of $D^{2}(X)$ onto the
subspace $Z$, implying that

$p\equiv ir:D^{2}(X)\rightarrow D^{2}(X)$

\noindent is a continuous linear projection ($p^{2}=p$) along $N(p)=N(r)$
onto $R(p)=R(r)=Z$. This implies that every such $Z$ admits a closed direct
complement in $D^{2}(X)$. Finally, in order to get a contradiction, an
appropriate pair $X$, $Z$ of concrete Banach spaces is needed. Let $X=c_{0}$
(the subspace of $l_{\infty }$ consisting of all null-convergent sequences
in $F$). Recall that $c_{0}$ is not bidual-like because of $%
D^{2}(c_{0})\cong l_{\infty }\not\cong c_{0}$. Namely, there are (isometric)
isomorphisms $D(c_{0})\cong l_{1}$ and $D(l_{1})\cong l_{\infty }$. However,
it is well known that there is a closed subspace $Z\trianglelefteq l_{\infty
}$, $c_{0}\trianglelefteq Z\trianglelefteq l_{\infty }\cong D^{2}(c_{0})$,
which does not admit any closed direct complement in $l_{\infty }$ - a
contradiction. This completes the proof.
\end{proof}

\begin{remark}
\label{R1}(i) Though, by Lemma 2 (ii), there are Banach spaces that are not
isomorphic to any of the dual spaces, we do not know whether every Banach
space is a retract of its second dual space (the converse of Lemma 2 (i)
(?)).

\noindent (ii) Since $D(F^{n})\cong F^{n}$ and, for all $1<p,q<\infty $ such
that $p^{-1}+q^{-1}=1$,

$D(F_{0}^{\mathbb{N}},\left\Vert \cdot \right\Vert _{p})=D(Cl_{l_{p}}(F_{0}^{%
\mathbb{N}},\left\Vert \cdot \right\Vert _{p}))=D(l_{p})\cong l_{q}$,

$D(l_{1})\cong l_{\infty }$ \quad and

$Cl_{l_{\infty }}(F_{0}^{\mathbb{N}},\left\Vert \cdot \right\Vert _{\infty
}))=c_{0}$, \quad implying

$D(F_{0}^{\mathbb{N}},\left\Vert \cdot \right\Vert _{\infty })=D(c_{0})\cong
l_{1}$

\noindent (see Lemma 4.1 (i) of [16]), the following (fundamental) question
occurs; Does the functor $D$ rise an uncountably infinite algebraic
dimension? In the next section we answer the question in negative.
\end{remark}

Lemma 2 motivates the following consideration. Given a normed space $X$ and
a $k\in \{0\}\cup \mathbb{N}$, let us denote by $j_{k,X}\equiv
j_{D^{k}(X)}:D^{k}(X)\rightarrow D^{k+2}(X)$ the canonical embedding ($%
D^{0}=!_{\mathcal{N}_{F}}$). Then the class $\{j_{k,X}\mid X\in Ob(\mathcal{N%
}_{F})\}$ determine a natural transformation $j_{k}:D^{k}\rightsquigarrow
D^{k+2}$ of the functors. When there is no ambiguity, i.e., when a normed
space $X$ is fixed, we simplify the notation $j_{k,X}$ to $j_{k}$. Notice
that, for a given $X\in Ob(\mathcal{N}_{F})$, the following morphisms of $%
\mathcal{B}_{F}\subseteq \mathcal{N}_{F}$ occur:

$D(X)\overset{j_{1}}{\rightarrow }D^{3}(X)\overset{D(j_{0})}{\rightarrow }%
D(X)$,

$D^{2}(X)%
\begin{array}{c}
\overset{j_{2}}{\rightarrow } \\ 
\underset{D^{2}(j_{0})}{\rightarrow }%
\end{array}%
D^{4}(X)\overset{D(j_{1})}{\rightarrow }D^{2}(X)$,

$D^{3}(X)%
\begin{array}{c}
\overset{j_{3}}{\rightarrow } \\ 
\underset{D^{2}(j_{1})}{\rightarrow }%
\end{array}%
D^{5}(X)%
\begin{array}{c}
\overset{D^{3}(j_{0})}{\rightarrow } \\ 
\underset{D(j_{2})}{\rightarrow }%
\end{array}%
D^{3}(X)$,

\noindent and, generally, for every $k\in \{0\}\cup \mathbb{N}$ and each $l$%
, $0\leq l\leq k$,

$D^{2k+1}(X)\overset{D^{2k-2l}(j_{2l+1})}{\rightarrow }D^{2k+3}(X)\overset{%
D^{2l+1-2l}(j_{2l})}{\rightarrow }D^{2k+1}(X)$,

$D^{2k+2}(X)\overset{D^{2k-2l}(j_{2l+2})}{\rightarrow }D^{2k+4}(X)\overset{%
D^{2k+1-2l}(j_{2l+1})}{\rightarrow }D^{2k+2}(X)$.

\noindent Let, for each $k$, $S_{2k+1}(X)$ be the set of all $%
D^{2k-2l}(j_{2l+l})\in L(D^{2k+1}(X),D^{2k+3}(X))$, and let $R_{2k+1}(X)$ be
the set of all $D^{2k+1-2l}(j_{2l})\in L(D^{2k+3}(X),D^{2k+1}(X))$, $0\leq
l\leq k$. Similarly, let $S_{2k+2}(X)$ be the set of all $%
D^{2k-2l}(j_{2l+2})\in L(D^{2k+2}(X),D^{2k+4}(X))$, and let $R_{2k+2}(X)$ be
the set of all $D^{2k+1-2l}(j_{2l+1})\in L(D^{2k+4}(X),D^{2k+2}(X))$, $0\leq
l\leq k$. Hence, for each $n\in \mathbb{N}$, the sets $S_{n}(X)$ and $%
R_{n}(X)$ are well defined. By Lemma 1, since all $j_{k}$ are isometries,
all the morphisms belonging to $S_{n}\cup R_{n}$ have norm $1$.

\begin{theorem}
\label{T1}Let $X$ be a normed space and let $n\in \mathbb{N}$. Then

\noindent (i) each $s\in S_{n}(X)$ is a (category) section, and each $r\in
R_{n}(X)$ is a (category) retraction;

\noindent (ii) $(\forall s\in S_{n}(X))(\exists r_{s}\in R_{n}(X))$ $%
r_{s}s=1_{D^{n}(X)}$;

\noindent (iii) $(\forall r\in R_{n}(X))(\exists s_{r}\in S_{n}(X))$ $%
rs_{r}=1_{D^{n}(X)}$;

\noindent (iv) $(\forall n\geq 3)(\exists s\in S_{n}(X))(\exists r\in
R_{n}(X))$ $rs$ is not an epimorphism (especially, $rs\neq 1_{D^{n}(X)\text{.%
}}$).
\end{theorem}

\begin{proof}
Let $X$ be a normed space. Since, for every $k\in \{0\}\cup \mathbb{N}$, the
canonical morphism $j_{k}:D^{k}(X)\rightarrow D^{k+2}(X)$ ($D^{0}=1_{%
\mathcal{N}_{F}}$) is an isometric embedding, Lemma 1 implies that $%
D(j_{k}):D^{k+3}(X)\rightarrow D^{k+1}(X)$ is an (open) epimorphism.
Statements (i) and (ii) can be proved by \textquotedblleft
parallel\textquotedblright\ induction on $2n-1$ and on $2n$. Nevertheless,
we provide an explicit proof. Let $n=1\in \mathbb{N}$. By Lemma 2, $%
D(j_{0})j_{1}=1_{D(X)}$, i.e., $j_{1}[D(X)]$ is a retract of $D^{3}(X)$ with
the retraction $D(j_{0})$ and the corresponding section $j_{1}$. Let $n\geq
2.$ Since $D$ is a contravariant functor, it follows that

$D(j_{1})D^{2}(j_{0})=D(D(j_{0})j_{1})=D(1_{D(X)})=1_{D^{2}(X)}$.

\noindent Therefore, $D^{2}(j_{0})[D^{2}(X)]$ is a retract of $D^{4}(X)$
with the retraction $D(j_{1})$ having $D^{2}(j_{0})$ for a corresponding
section. In general, by considering $D^{n}(X)$ as $D(D^{n-1}(X))$, i.e., the
canonical embedding $j_{n}:D^{n}(X)\rightarrow D^{n+2}(X)$ as
\textquotedblleft $j_{1}:D(D^{n-1}(X)\rightarrow D^{3}(D^{n+1}(X))$%
\textquotedblright , and $D(j_{n-1}):D^{n+2}(X)\rightarrow D^{n}(X)$ as
\textquotedblleft $D(j_{0}):D^{3}(D^{n-1}(X))\rightarrow D(D^{n-1}(X))$%
\textquotedblright , one proves (by mimicking the appropriate part of the
proof of Lemma 2) that

$D(j_{n-1})j_{n}=1_{D^{n}(X)}$.

\noindent holds true. Thus, for every $n\in \mathbb{N}$, $j_{n}[D^{n}(X)]$
is a retract of $D^{n+2}(X)$ with the retraction $D(j_{n-1})$ having $j_{n}$
for a corresponding section. Further, since $D^{2}$ is a (covariant)
functor, one readily verifies that, for every $k\in \mathbb{N}$ and every $%
l\in \{0,\ldots ,k\}$,

$%
D^{2k+1-2l}(j_{2l})D^{2k-2l}(j_{2l+1)}=D^{2k-2l}(D(j_{2l})j_{2l+1})=D^{2k-2l}(1_{D^{2l+1}(X)})=1_{D^{2k+1}(X)} 
$,

$%
D^{2k+1-2l}(j_{2l+1})D^{2k-2l}(j_{2l+2})=D^{2k-2l}(D(j_{2l+1})j_{2l+2})=D^{2k-2l}(1_{D^{2l+2}(X)})=1_{D^{2k+2}(X)} 
$.

\noindent This shows that all $D^{2k-2l}(j_{2l+1})$ and $D^{2k-2l}(j_{2l+2})$
are sections having $D^{2l+1-2l}(j_{2l})$ and $D^{2l+1-2l}(j_{2l+1})$ for
the corresponding retractions, respectively, and vice versa. Therefore,
statements (i), (ii) and (iii) hold true.

\noindent Concerning statement (iv), let firstly $n=3$. We are to show that,
for the section $s=j_{3}:D^{3}(X)\rightarrow D^{5}(X)$ and the retraction $%
r=D^{3}(j_{0}):D^{5}(X)\rightarrow D^{3}(X)$, the composite

$rs=D^{3}(j_{0})j_{3}:D^{3}(X)\rightarrow D^{3}(X)$

\noindent may be not an epimorphism. Consider the following diagram

$%
\begin{array}{ccc}
D^{3}(X) & \overset{D(j_{0})}{\rightarrow } & D(X) \\ 
j_{3}\downarrow &  & \downarrow j_{1} \\ 
D^{5}(X) & \underset{D^{3}(j_{0})}{\rightarrow } & D^{3}(X)%
\end{array}%
$

\noindent in $\mathcal{B}_{F}\subseteq \mathcal{N}_{F}$. Since $D^{3}=D^{2}D$%
, $D^{5}=D^{2}D^{3}$, $j_{1}=j_{D(X)}$, $j_{3}=j_{D^{3}(X)}$ and $j:1_{%
\mathcal{N}_{F}}\rightsquigarrow D^{2}$ is a natural transformation of the
functors, the diagram commutes, i.e., $D^{3}(j_{0})j_{3}=j_{1}D(j_{0})$.
Notice that, in general, the canonical embedding $j_{1}$ is not an
epimorphism, and the conclusion follows. If $n=4$, then one similarly proves
that, for instance, $D^{3}(j_{1})j_{4}$ is not an epimorphism. Generally, if
an $rs:D^{n}(X)\rightarrow D^{n}(X)$ factorizes through a $j_{n-2k}$, $1\leq
k\leq n-2$, then, generally, it is not an epimorphism. Thus, statement (iv)
follows.
\end{proof}

The next corollary is an immediate consequence of our Theorem 1 and the
known general facts (see Section 6 of Chapter 6 of [6].

\begin{corollary}
\label{C1}(i) For every normed space $X$ and each $n\in \mathbb{N}$, the
range $R(j_{n})$ and $the$ annihilator $R(j_{n-1})^{0}$ (of $R(j_{n-1})$
with respect to $D^{n+2}(X)$) are closed complementary subspaces of $%
D^{n+2}(X)$, i.e., by identifying $D^{n-1}(X)$ with $R(j_{n-1})$ and $%
D^{n}(X)$ with $R(j_{n})$, the closed direct-sum presentation

$D^{n+2}(X)=D^{n}(X)\dotplus D^{n-1}(X)^{0}$

\noindent holds true. Consequently, by assuming the mentioned
identifications,

$D^{2n+1}(X)\cong D(X)\dotplus X^{0}\dotplus D^{2}(X)^{0}\dotplus \cdots
\dotplus D^{2n-2}(X)^{0}$ \quad and

$D^{2n+2}(X)\cong D^{2}(X)\dotplus D(X)^{0}\dotplus D^{3}(X)^{0}\dotplus
\cdots \dotplus D^{2n-1}(X)^{0}$.

\noindent (ii) If $X$ is a normed space admitting a retraction $%
r_{0}:D^{2}(X)\rightarrow Cl(j_{0}[X])\equiv \bar{X}$ (in $\mathcal{B}_{F}$%
), then

$D^{2}(X)=\bar{X}\dotplus N(r_{0})$

\noindent is a closed direct-sum presentation of $D^{2}(X)$. If, in
addition, $X$ is a Banach space, then $D^{2}(X)=X\dotplus N(r_{0})$ is a
closed direct-sum presentation.
\end{corollary}

\begin{proof}
(i). Notice that, for a given $X$ and each $n\in \mathbb{N}$,

$p_{n+2}\equiv j_{n}D(j_{n-1}):D^{n+2}(X)\rightarrow D^{n+2}(X)$

\noindent is a continuous linear projection. Indeed,

$%
p_{n+2}^{2}=(j_{n}D(j_{n-1}))(j_{n}D(j_{n-1}))=j_{n}1_{D^{n}(X)}D(j_{n-1})=j_{n}D(j_{n-1})=p_{n+2} 
$.

\noindent Since $D(j_{n-1})$ is an epimorphism. and $j_{n}$ is an isometric
embedding, it follows that

$R(p_{n+2})=R(j_{n})\cong D^{n}(X)$.

\noindent Further,

$N(p_{n+2})=N(D(j_{n-1}))=\{x^{n+2}\in D^{n+2}(X)\mid
x^{n+2}j_{n-1}=c_{0}^{n-1}\}=$

$=(j_{n-1}[D^{n-1}(X)])^{0}\equiv R(j_{n-1})^{0}$.

\noindent Now the conclusion follows by induction and the well known general
facts. (Observe that, for instance,

$p_{n+2}^{\prime }\equiv D^{2}(j_{n-2})D(j_{n-1}):D^{n+2}(X)\rightarrow
D^{n+2}(X)$, $n>1$, ,

\noindent is also a continuous linear projection yielding another closed
direct-sum presentation of $D^{n+2}(X)$.)

\noindent (ii). If $X$ admits a retraction $r_{0}:D^{2}(X)\rightarrow
Cl(j_{0}[X])\equiv \bar{X}$, then

$p_{2}\equiv j_{0}r_{0}:D^{2}(X)\rightarrow D^{2}(X)$

\noindent is a continuous linear projection. Since the both $%
R(p_{2})=Cl(R(j_{0}))\equiv \bar{X}$ and $N(p_{2})=N(r_{0})$ are closed in $%
D^{2}(X)$, the stated closed direct-sum presentation follows. If such an $X$
is a Banach space, one may identify $X\equiv \bar{X}\subseteq D^{2}(X)$, and
the conclusion follows.
\end{proof}

\begin{example}
\label{E1}Recall that $D(c)\cong l_{1}\cong D(c_{0})$ and $D(l_{1})\cong
l_{\infty }$. Then, by Corollary 1, $D(l_{\infty })\cong l_{1}\dotplus
c^{0}\cong l_{1}\dotplus c_{0}^{0}$. This also shows that the annihilator
does not preserve separability of a subspace.
\end{example}

By the mentioned identifications, Corollary 1 shows that $%
D^{n+2}(X)/D^{n}(X) $ is (isometrically) isomorphic to $D^{n-1}(X)^{0}$.
Further, it is well known that $D(X)/Z^{0}$, $Z\trianglelefteq X$, is
isometrically isomorphic to $D(Z)$, and that, for Banach spaces, $D(X/Z)$ is
isometrically isomorphic to $Z^{0}$. These facts and Theorem 1 aim our
attention at the behavior of the iterated dual functors on a quotient space
(see also [6], Chapter 6., Sections 5. and 6.).

\begin{lemma}
\label{L3}Let $Z$ be a closed subspace of a normed space $X$, and denote by $%
i:Z\hookrightarrow X$ the inclusion, and by $q:X\rightarrow X/Z$ the
quotient morphism. Then, for every $n\in \mathbb{N}$, the short sequences

$D^{2n-1}(Z)\overset{D^{2n-1}(i)}{\leftarrow }D^{2n-1}(X)\overset{D^{2n-1}(q)%
}{\leftarrow }D^{2n-1}(X/Z)$ \quad and

$D^{2n}(Z)\overset{D^{2n}(i)}{\rightarrow }D^{2n}(X)\overset{D^{2n}(q)}{%
\rightarrow }D^{2n}(X/Z)$

\noindent in $\mathcal{B}_{F}$ are exact.
\end{lemma}

\begin{proof}
Clearly, the short sequence

$Z\overset{i}{\hookrightarrow }X\overset{q}{\rightarrow }X/Z$

\noindent in $\mathcal{N}_{F}$ is exact, i.e., $R(i)=N(q)$. Since $D$ is a
contravariant functor and the function $D_{X/Z}^{Z}:L(X,X/Z)\rightarrow
L(D(X/Z),D(X))$ is linear, it follows that

$D(i)D(q)=D(qi)=D(c_{\theta })=c_{\theta }^{1}$,

\noindent i.e., $R(D(q))\subseteq N(D(i))$. We are to prove that the
converse $N(D(i))\subseteq R(D(q))$ holds as well. Let $x^{1}\in
N(D(i))\subseteq X^{\ast }$, i.e., $D(i)(x^{1})=x^{1}i=c_{0}$ which implies
that $x^{1}[Z]=\{0\}$, i.e., $x^{1}\in Z^{0}$. By the universal property of
the quotient morphism $q$, there exists a continuous linear function

$w_{x^{1}}:X/Z\rightarrow F$, $w_{x^{1}}([x])\equiv w_{x^{1}}q(x)=x^{1}(x)$.

\noindent Then, clearly, $w_{x^{1}}\in (X/Z)^{\ast }$ and, moreover,

$D(q)(w_{x^{1}})=w_{x^{1}}q=x^{1}$,

\noindent implying that $x^{1}\in R(D(q))$, which proves the converse.
Hence, the short sequence

$D(Z)\overset{D(i)}{\leftarrow }D(X)\overset{D(q)}{\leftarrow }D(X/Z)$

\noindent in $\mathcal{B}_{F}$ is exact. Further, by Lemma 1, $%
D(i):D(X)\rightarrow D(Z)$ is an epimorphism., and thus, the range $%
R(D(i))=D(Z)$ is (trivially) closed in $D(Z)$. Then, by Proposition 6.5.13.
of [6], the short sequence

$D^{2}(Z)\overset{D^{2}(i)}{\rightarrow }D^{2}(X)\overset{D^{2}(q)}{%
\rightarrow }D^{2}(X/Z)$

\noindent in $\mathcal{B}_{F}$ is exact. Now, in general, by Lemma 1, $%
D^{2n}(q)$ and $D^{2n+1}(i)$ are epimorphisms, i.e., $%
R(D^{2n}(q))=D^{2n}(X/Z)$ and $R(D^{2n+1}(i))=D^{2n+1}(Z)$. (It suffices
that $R(D^{2n}(q)$ is closed in $D^{2n}(X/Z)$ and that $R(D^{2n+1}(i))$ is
closed in $D^{2n+1}(Z)$, which follows by Proposition 6.5.12. of [6].) Then
the final conclusion follows by Proposition 6.5.13. of [6].
\end{proof}

We can now state the following general facts concerning the iterated dual
functors and quotients.

\begin{theorem}
\label{T2}Let $Z$ be a closed subspace of a normed space $X$, and denote by $%
i:Z\hookrightarrow X$ the inclusion, and by $q:X\rightarrow X/Z$ the
quotient morphism. Then, for each $n\in \mathbb{N}$,

\noindent (i) the functor $D^{2n-1}$ permits cancellation on the quotient
objects, i.e.,

$D^{2n-1}(X)/D^{2n-1}(X/Z))\cong D^{2n-1}(Z)$,

\noindent where $D^{2n-1}(X/Z)$ is identified with $R(D^{2n-1}(q))$ in $%
D^{2n-1}(X)$;

\noindent (ii) the functor $D^{2n}$ \textquotedblleft
preserves\textquotedblright\ the quotient of objects, i.e.,

$D^{2n}(X/Z)\cong D^{2n}(X)/D^{2n}(Z)$,

\noindent where $D^{2n}(Z)$ is identified with $R(D^{2n}(i))$ in $D^{2n}(X)$.

\noindent (iii) $D(X/Z)\cong Z^{0}$, $D^{2n+1}(X/Z)\cong D^{2n}(Z)^{0}$, $%
D^{2n}(Z)\cong D^{2n-1}(X/Z)^{0}$,

\noindent where the mentioned identifications are assumed. Furthermore, all
the isomorphisms are isometric.
\end{theorem}

\begin{proof}
Consider the short sequence

$Z\overset{i}{\hookrightarrow }X\overset{q}{\rightarrow }X/Z$, $\quad
R(i)\equiv i[Z]=N(q)$,

\noindent in $\mathcal{N}_{F}$, which is exact, $R(i)=Z=N(q)$, and $i$ is
the closed inclusion, while $q$ is an open epimorphism. By Lemma 3, for each 
$n\in \mathbb{N}$, the sequences

$D^{2n-1}(Z)\overset{D^{2n-1}(i)}{\leftarrow }D(^{2n-1}(X)\overset{%
D(^{2n-1}q)}{\leftarrow }D^{2n-1}(X/Z)$, \quad and

$D^{2n}(Z))\overset{D^{2n}(i)}{\rightarrow }D^{2n}(X)\overset{D^{2n}(q)}{%
\rightarrow }D^{2n}(X/Z)$,

\noindent in $\mathcal{B}_{F}$ are exact, i.e.,

$R(D^{2n-1}(q))\equiv D^{2n-1}(q)[D^{2n-1}(X/Z)])=N(D^{2n-1}(i))$,

$R(D^{2n}(i))\equiv D^{2n}(i)[D^{2n}(Z)]=N(D^{2n}(q))$

\noindent Observe that, by Lemma 1, the morphisms $D^{2n-1}(q)$ and $%
D^{2n}(i)$ are closed monomorphisms, while $D^{2n-1}(i)$ and $D^{2n}(q)$ are
(open) epimorphisms. By the universal property of a quotient in $\mathcal{N}%
_{F}$, there exists a unique continuous linear (canonical) factorization of $%
D^{2n-1}(i)$ trough the quotient morphism

$q_{2n-1}:D^{2n-1}(X)\rightarrow D^{2n-1}(X)/N(D^{2n-1}(i))$,

$q_{2n-1}(x^{2n-1})=[x^{2n-1}]$,

\noindent such that $D^{2n-1}(i)=h_{2n-1}q_{2n-1}\in Mor(\mathcal{B}_{F})$,
where

$h_{2n-1}:D^{2n-1}(X)/N(D^{2n-1}(i))\rightarrow D^{2n-1}(Z)$,

$h_{2n-1}([x^{2n-1}])=D^{2n-1}(i)(x^{2n-1})=x^{2n-1}D^{2n-2}(i)$.

\noindent By the same reason, there exists the canonical factorization $%
D^{2n}(q)=h_{2n}q_{2n}$, where

$q_{2n}:D^{2n}(X)\rightarrow D^{2n}(X)/N(D^{2n}(q))$, $%
q_{2n}(x^{2n})=[x^{2n}]$, \quad and

$h_{2n}:D^{2n}(X)/N(D^{2n}(i))\rightarrow D^{2n}(X/Z)$,

$h_{2n}([x^{2n}])=D^{2n}(q)(x^{2n})=x^{2n}D^{2n-1}(q)$.

\noindent Since $D^{2n-1}(i)$ and $D^{2n}(q)$ are open epimorphisms, so are $%
h_{2n-1}$ and $h_{2n}$ (Open-mapping theorem). Further, the above exactness,
i.e.,

$R(D^{2n-1}(q))=N(D^{2n-1}(i))$, $R(D^{2n}(i))=N(D^{2n}(q)$

\noindent imply, respectively, that

$D^{2n-1}(X)/R(D^{2n-1}(q))=D^{2n-1}(X)/N(D^{2n-1}(i))$,

$D^{2n}(X)/R(D^{2n}(i))=D^{2n}(X)/N(D^{2n}(q))$.

\noindent Therefore, $h_{2n-1}$ and $h_{2n}$ are bijections. Finally, by the
Banach inverse-mapping theorem, $h_{2n-1}$ and $h_{2n}$ are isomorphisms of $%
\mathcal{B}_{F}$. Since $D^{2n-1}(q)$ and $D^{2n}(i)$ are closed
monomorphisms, one may identify $D^{2n-1}(X/Z)$ with $%
D^{2n-1}(q)[D^{2n-1}(X/Z)]$ in $D^{2n-1}(X)$ as well as $D^{2n}(Z)$ with $%
D^{2n}(i)[D^{2n}(Z)]$ in $D^{2n}(X)$. Consequently,

$D^{2n-1}(X)/D^{2n-1}(X/Z))\cong D^{2n-1}(Z)$ \quad and

$D^{2n}(X)/D^{2n}(Z))\cong D^{2n}(X/Z)$.

\noindent In this way we have proven the isomorphism relations in statements
(i) and (ii).

\noindent In order to prove statement (iii) (see also Remark 2 below), let
us again consider the starting exact sequence in $\mathcal{N}_{F}$, $%
R(i)=Z=N(q)$. Notice that

$Z^{0}=\{x^{1}\in X^{\ast }\mid R(x^{1})=\{0\}\}=\{x^{1}\mid
x^{1}i=c_{0}\}=N(D(i))=R(D(q))$,

\noindent implying that

$R(D(q))\equiv D(q)[(X/Z)^{\ast }]=Z^{0}$ in $X^{\ast }$.

\noindent Since, by Lemma 1, $D(q)$ is a closed monomorphism, one may
identify $(X/Z)^{\ast }$ with $Z^{0}$ in $X^{\ast }$, and then (the well
known) $D(X/Z)\cong Z^{9}$ holds. This proves the case $n=1$ of statement
(iii). If $n=2$, the exactness (Lemma 3) and Lemma 1 imply that $%
R(D^{2}(i))=((X/Z)^{\ast })^{0}$ and $D^{2}(i)$ is a closed monomorphism.
Then, by identifying $Z^{\ast \ast }$ with $((X/Z)^{\ast })^{0}$ in $X^{\ast
\ast }$, it follows $D^{2}(Z)\cong D(X/Z)^{0}$. By arguing in the same
manner through all the $D$-iterating exact sequences (Lemma 3), and assuming
the mentioned identifications (Lemma 1), one obtains the remaining
isomorphisms in (iii). Finally, since the isomorphisms $X^{\ast }/(X^{\ast
}/Z^{\ast })\cong X^{\ast }/Z^{0}\cong Z^{\ast }$ and $(X/Z)^{\ast \ast
}\cong X^{\ast \ast }/Z^{\ast \ast }$ are isometric (a non-zero quotient
morphism has the norm $1$), it follows, by applying the functor $D^{2n}$
inductively, that all the obtained isomorphisms are isometric.
\end{proof}

\begin{remark}
\label{R2}We are aware of the well known fact (closely related to the case $%
n=1$ of Theorem 2, (i) and (iii)), that $X^{\ast }/Y^{0}\cong Y^{\ast }$
(isometrically), for \emph{every} subspace $Y$ of $X$ ([10], Section 8. 12,
Propozicija 17, p. 444). We did not use it in the proof. However, one can
show that statements (i) and (ii) of Theorem 2 with that fact imply (iii),
and conversely, statement (iii) with that fact implies (i) and (ii) of
Theorem 2.
\end{remark}

\section{An application through quotient shapes}

We shall now combine the obtained facts with those of [14], [16] and [17] in
order to get a better insight in the quotient shapes of normed spaces
(especially those considered in [16] and [17]) related to their (iterated)
dual spaces. The first step in that direction is based on the fact that $%
D^{2}$ is a faithful functor.

\begin{theorem}
\label{T3}(i) Let $X$ be a normed space and let

$\boldsymbol{p}_{\text{\b{0}}}=(p_{\lambda }):X\rightarrow \boldsymbol{X}_{%
\text{\b{0}}}=(X_{\lambda },p_{\lambda \lambda ^{\prime }},\Lambda _{\text{%
\b{0}}})$

\noindent be an $(\mathcal{N}_{F})_{\text{\b{0}}}$-expansion of $X$. Then

$pro$-$D^{2}(\boldsymbol{p}_{\text{\b{0}}})=(D^{2}(p_{\lambda
})):D^{2}(X)\rightarrow pro$-$D^{2}(\boldsymbol{X}_{\text{\b{0}}})=$

$=(D^{2}(X_{\lambda }),D^{2}(p_{\lambda \lambda ^{\prime }}),\Lambda _{\text{%
\b{0}}})$

\noindent is an $D^{2}(\mathcal{N}_{F})_{\text{\b{0}}}$-expansion of $%
D^{2}(X)$.

\noindent (ii) Let $X$ and $Y$ be normed spaces of the same finite quotient
shape type, i.e., $Sh_{\text{\b{0}}}(X)=Sh_{\text{\b{0}}}(Y)$. Then, for
every $n\in \mathbb{N}$, $Sh_{\text{\b{0}}}(D^{2n}(X))=Sh_{\text{\b{0}}%
}(D^{2n}(Y))$, i.e., $D^{2n}$ preserves the finite quotient shape type.
\end{theorem}

\begin{proof}
(i). According to Lemma 1 (i), given an $(\mathcal{N}_{F})_{\text{\b{0}}}$%
-expansion (actually, a $(\mathcal{B}_{F})_{\text{\b{0}}}$-expansion)

$\boldsymbol{p}_{\text{\b{0}}}=(p_{\lambda }):X\rightarrow \boldsymbol{X}_{%
\text{\b{0}}}=(X_{\lambda },p_{\lambda \lambda ^{\prime }},\Lambda _{\text{%
\b{0}}})$

\noindent of $X$, one has to verify the factorization property(E1) of

$(D^{2}(p_{\lambda })):D^{2}(X)\rightarrow (D^{2}(X_{\lambda
}),D^{2}(p_{\lambda \lambda ^{\prime }}),\Lambda _{\text{\b{0}}})$

\noindent with respect to the image subcategory $D^{2}(\mathcal{N}_{F})_{%
\text{\b{0}}}$ only. Let $D^{2}(Y)$ be a finite-dimensional normed space
(actually, a Banach space $Z\cong F^{n}$, where $n\in \mathbb{N}$) and let a
morphism $D^{2}(f)\in \mathcal{N}_{F}(D^{2}(X),D^{2}(Y))$ be given. Then $Y$
is finite-dimensional and $f\in \mathcal{N}_{F}(X,Y)$. Since $\boldsymbol{p}%
_{\text{\b{0}}}:X\rightarrow \boldsymbol{X}_{\text{\b{0}}}$ is an $(\mathcal{%
N}_{F})_{\text{\b{0}}}$-expansion of $X$, there exit a $\lambda \in \Lambda
_{\text{\b{0}}}$ and an $f_{\lambda }:X_{\lambda }\rightarrow Y$ such that $%
f_{\lambda }p_{\lambda }=f$. Then $D^{2}(f_{\lambda })D^{2}(p_{\lambda
})=D^{2}(f_{\lambda }p_{\lambda })=D^{2}(f)$, and the claim follows.

\noindent (ii). It suffices to prove the claim in the case $n=2$. Let

$\boldsymbol{p}_{\text{\b{0}}}=(p_{\lambda }):X\rightarrow \boldsymbol{X}_{%
\text{\b{0}}}=(X_{\lambda },p_{\lambda \lambda ^{\prime }},\Lambda _{\text{%
\b{0}}})$,

$\boldsymbol{q}_{\text{\b{0}}}=(q_{\mu }):Y\rightarrow \boldsymbol{Y}_{\text{%
\b{0}}}=(Y_{\mu },q_{\mu \mu ^{\prime }},M_{\text{\b{0}}})$

\noindent be any $(\mathcal{N}_{F})_{\text{\b{0}}}$-expansions (actually, ($%
\mathcal{B}_{F})_{\text{\b{0}}}$-expansions) of $X$, $Y$ respectively. Since 
$Sh_{\text{\b{0}}}(X)=Sh_{\text{\b{0}}}(Y)$, the expansion systems $%
\boldsymbol{X}_{\text{\b{0}}}$ and $\boldsymbol{Y}_{\text{\b{0}}}$ are
isomorphic objects of $pro$-$(\mathcal{N}_{F})_{\text{\b{0}}}$. Then $pro$-$%
D^{2}(\boldsymbol{X}_{\text{\b{0}}})$ and $pro$-$D^{2}(\boldsymbol{Y}_{\text{%
\b{0}}})$ are isomorphic objects of $pro$-$D^{2}(\mathcal{N}_{F})_{\text{\b{0%
}}})$, because

$pro$-$D^{2}:pro$-$(\mathcal{N}_{F})_{\text{\b{0}}}\rightarrow pro$-$D^{2}(%
\mathcal{N}_{F})_{\text{\b{0}}}$

\noindent is a functor (the restriction of the \textquotedblleft
prolongation\textquotedblright\ of $D^{2}$ to the pro-categories).Since $pro$%
-$D^{2}(\mathcal{N}_{F})_{\text{\b{0}}}$ is a subcategory of $pro$-$(%
\mathcal{N}_{F})_{\text{\b{0}}}$, it follows that $D^{2}(\boldsymbol{X}_{%
\text{\b{0}}})$ and $D^{2}(\boldsymbol{Y}_{\text{\b{0}}})$ are isomorphic in 
$pro$-$(\mathcal{N}_{F})_{\text{\b{0}}}$ as well. Then, by (i), $Sh_{\text{%
\b{0}}}(D^{2}(X))=Sh_{\text{\b{0}}}(D^{2}(Y))$.
\end{proof}

\begin{corollary}
\label{C2}All the $l_{p}$ spaces and all their normed duals belong to the
same finite quotient shape type, i.e.,

$(\forall 1\leq p\leq \infty )(\forall n\in \{0\}\cup \mathbb{N})$ $Sh_{%
\text{\b{0}}}(D^{n}(l_{p}))=Sh_{\text{\b{0}}}(F_{0}^{\mathbb{N}},\left\Vert
\cdot \right\Vert _{2})$

\noindent The same holds for all the $L_{p}(n)$ spaces.
\end{corollary}

\begin{proof}
Recall that, byTheorem 2 (i) of [17], $Sh_{\text{\b{0}}}(l_{p})=Sh_{\text{\b{%
0}}}(l_{p^{\prime }})$ holds for all $1\leq p,p^{\prime }\leq \infty $.
Since the spaces l$_{p}$, $1<p<\infty $, are reflexive, Theorem 3 implies
that all the even normed duals of all $l_{p}$ spaces belong to the same
finite quotient shape type. Especially,

$Sh_{\text{\b{0}}}(D^{2n}(l_{1}))=Sh_{\text{\b{0}}}(D^{2n}(l_{\infty }))=Sh_{%
\text{\b{0}}}(l_{2})=Sh_{\text{\b{0}}}(F_{0}^{\mathbb{N}},\left\Vert \cdot
\right\Vert _{2})$.

\noindent Further, since $l_{p}\cong l_{p^{\prime }}$, for $1<p,p^{\prime
}<\infty $ and $p^{-1}+(p^{\prime })^{-1}=1$, and since $l_{\infty }\cong
D(l_{1})$, it follows that all the $l_{p}$ spaces and all their normed duals
belong to the same finite quotient shape type, i.e.,

$(\forall 1\leq p\leq \infty )(\forall n\in \{0\}\cup \mathbb{N})$ $Sh_{%
\text{\b{0}}}(D^{n}(l_{p}))=Sh_{\text{\b{0}}}(F_{0}^{\mathbb{N}},\left\Vert
\cdot \right\Vert _{2})$

\noindent Then, by Corollary 1 of [17] and Theorem 3, the same holds for all
the $L_{p}(n)$ spaces.
\end{proof}

Now, the question about the \emph{quotient shapes} of any normed (Banach,
separable) space, occurs as the \emph{problem of the algebraic dimension(s)
of its (iterated) normed dual space(s)}. This obstacle has essentially
limited the obtained results of [17]. Namely, the restriction to separable
or to bidual-like normed spaces (Theorems 2 and 4 of [17]) was necessary
because, in essence, we did not know how to calculate $\dim D(X)$ (except
for spaces of the mentioned classes, see Lemma 4 (ii) of [17]). We have
hereby resolved that problem completely. Firstly, an auxiliary notion.

\begin{definition}
\label{D1}A normed vectorial space $X$ over $F\in \{\mathbb{R},\mathbb{C}\}$
is said to be $\dim ^{\ast }$\textbf{-stable (or }$dD$\textbf{-stable)} if $%
\dim D(X)\equiv \dim X^{\ast }=\dim X$.
\end{definition}

Clearly, the functor $D$ does not diminish the algebraic dimension, and thus 
$\dim X\leq \dim D(X)$ holds generally.

\begin{example}
\label{E2}(i) Since $D(F^{n})\cong F^{n}$, $n\in \mathbb{N}$, every
finite-dimensional normed space over $F\in \{\mathbb{R},\mathbb{C}\}$ is $%
\dim ^{\ast }$-stable. Similarly, all separable and all bidual-like normed
spaces are $\dim ^{\ast }$-stable.

\noindent (ii) Since $l_{p}^{\ast }\cong l_{p^{\prime }},$ $1/p+1/p^{\prime
}=1$, all $l_{p}$ spaces, $1<p<\infty $, are $\dim ^{\ast }$-stable, while $%
l_{1}$ is $\dim ^{\ast }$-stable because $l_{1}^{\ast }\cong l_{\infty }$.
The same holds true for all $L_{p}(n)$ spaces, $n\in \mathbb{N}$.

\noindent (iii) No direct sum normed space $(F_{0}^{\mathbb{N}},\left\Vert
\cdot \right\Vert )$ is $\dim ^{\ast }$-stable. (Namely, $\dim (F_{0}^{%
\mathbb{N}},\left\Vert \cdot \right\Vert )=\aleph _{0}$, while $\dim (F_{0}^{%
\mathbb{N}},\left\Vert \cdot \right\Vert )^{\ast }\neq \aleph _{0}$ (every
dual space is a Banach space, and there is no Banach space having countably
infinite algebraic dimension).
\end{example}

Since the finite quotient shape is an invariant of the algebraic dimension
(Theorem 2 (i) of [17]), the next corollary follows immediately.

\begin{corollary}
\label{C3}If $X\in Ob(\mathcal{N}_{F})$ is $\dim ^{\ast }$-stable, then $Sh_{%
\text{\b{0}}}(X^{\ast })=Sh_{\text{\b{0}}}(X)$.
\end{corollary}

Recall that the algebraic dual rises every infinite algebraic dimension.
Thus, there is no countably infinite-dimensional algebraic dual space. Since
there is no countably infinite-dimensional Banach space, there is no
countably infinite-dimensional (algebraically) dual normed space as well.
However, besides Example 2, (i) and (ii), we shall show that the class of
all $\dim ^{\ast }$-stable naormed spaces is rather large. The main fact in
that direction is the next lemma.

\begin{lemma}
\label{L4}Let $X,Y\in Ob(\mathcal{N}_{F})$ such that $Sh_{\text{\b{0}}%
}(X)=Sh_{\text{\b{0}}}(Y)$. If $Y$ is $\dim ^{\ast }$-stable, i.e., $\dim
Y^{\ast }=\dim Y$, then so is $X^{\ast }\equiv D(X)$ and $\dim X^{\ast \ast
}=\dim X^{\ast }=\dim Y$.
\end{lemma}

\begin{proof}
Clearly, the statement is not trivial in the infinite-dimensional case only.
Since $Y$ is $\dim ^{\ast }$-stable, i.e., $\dim Y=\dim Y^{\ast }$, and $%
Y^{\ast }$ is a Banach space, it follows that $\dim Y=|Y|$ $\geq 2^{\aleph
_{0}}$ (see Lemma 3. 2 (iv) of [14]). Assume, firstly, that $\dim X\geq
2^{\aleph _{0}}$ as well. Then $\dim X=|X|$ and $\dim X^{\ast }=|X^{\ast }|$ 
$\geq |X|$. Let

$\boldsymbol{p}_{\text{\b{0}}}=(p_{\lambda }):X\rightarrow \boldsymbol{X}_{%
\text{\b{0}}}=(X_{\lambda },p_{\lambda \lambda ^{\prime }},\Lambda _{\text{%
\b{0}}})$,

$\boldsymbol{q}_{\text{\b{0}}}=(q_{\mu }):Y\rightarrow \boldsymbol{Y}_{\text{%
\b{0}}}=(Y_{\mu },q_{\mu \mu ^{\prime }},M_{\text{\b{0}}})$

\noindent be the canonical $(\mathcal{N}_{F})_{\text{\b{0}}}$-expansions
(actually, ($\mathcal{B}_{F})_{\text{\b{0}}}$-expansions) of $X$, $Y$
respectively. Since $\dim X\geq 2^{\aleph _{0}}$, the canonical construction
of $\boldsymbol{p}_{\text{\b{0}}}$ (see Section 12 of [13] and Section 4.1
of [14]) implies that the index set $\Lambda _{\text{\b{0}}}$ is the
disjoint union of $\Lambda _{\text{\b{0}}}^{(n)}$, $n\in \{0\}\cup \mathbb{N}
$, where $\Lambda _{\text{\b{0}}}^{(0)}$ is the singleton containing the
first (minimal) element, while, for each $n\in \mathbb{N}$,

$\Lambda _{\text{\b{0}}}^{(n)}=\{\lambda \in \Lambda _{\text{\b{0}}}\mid
\dim X_{\lambda }=n\}$, \quad and

\TEXTsymbol{\vert}$\Lambda _{\text{\b{0}}}^{(1)}|$ $=\cdots =|\Lambda _{%
\text{\b{0}}}^{(n)}|$ $=\cdots =|\Lambda _{\text{\b{0}}}$ $|\geq |X|$ $\geq
2^{\aleph _{0}}$.

\noindent Indeed, given an $n\in \mathbb{N}$, for every $\lambda \in \Lambda
_{\text{\b{0}}}^{(n)}$, $X_{\lambda }=X/Z_{\lambda }$ where $Z_{\lambda
}\trianglelefteq X$ is closed, $\dim Z_{\lambda }=\dim X$ and $\dim
(X/Z_{\lambda })=n$. Thus, there is a closed direct complement $W_{\lambda
}\trianglelefteq X$ of $Z_{\lambda }$, $X=Z_{\lambda }\dotplus W_{\lambda }$%
, $\dim W_{\lambda }=n$. Then, for $n=1$, each continuous linear
epimorphism. $x^{1}:X\rightarrow F$, i.e., $x^{1}\in X^{\ast }\setminus
\{c_{0}\}$, yields a unique $\lambda \in \Lambda _{\text{\b{0}}}^{(1)}$.
Conversely, for every $\lambda \in \Lambda _{\text{\b{0}}}^{(1)}$, it can
exist at most $2^{\aleph _{0}}\cdot |X|$ linear epimorphisms $x^{1}$ having $%
N(x^{1})=Z_{\lambda }$, where $|X|$ counts all the $1$-dimensional direct
complements $W_{\lambda }$ of $Z_{\lambda }$. Similarly, each continuous
linear epimorphism. $f:X\rightarrow F^{n}$ yields a unique closed $%
Z_{f}\equiv N(f)\trianglelefteq X$ such that $Z_{f}\dotplus W_{f}=X$, $%
W_{f}\cong R(f)=F^{n}$, while, since there are $2^{\aleph _{0}}$ linear
epimorphisms of $F^{n}$ to $F^{n}$, for every such $Z_{\lambda }$, it can
exist at most $2^{\aleph _{0}}\cdot |X|$ linear epimorphisms $f$ having $%
N(f)=Z_{\lambda }$, where $|X|$ counts all the $n$-dimensional direct
complements $W_{\lambda }$ of $Z_{\lambda }$. Now, one readily sees that all 
$\Lambda _{\text{\b{0}}}^{(n)}$, have the same uncountable cardinality.
Since $|\mathbb{N}|$ $=\aleph _{0}$ and $\dim X=|X|$ $\geq 2^{\aleph _{0}}$,
the above partition of $\Lambda _{\text{\b{0}}}$ follows by the cardinal
arithmetic. Further, especially observe that

$|\Lambda _{\text{\b{0}}}^{(1)}|$ $\leq |X^{\ast }|$ $\leq |\Lambda _{\text{%
\b{0}}}^{(1)}|\cdot |X|\cdot 2^{\aleph _{0}}=|\Lambda _{\text{\b{0}}}^{(1)}|$
$=|\Lambda _{\text{\b{0}}}|$.

\noindent Therefore,

$|\Lambda _{\text{\b{0}}}|$ $=|X^{\ast }|$ $=\dim X^{\ast }\geq \dim X=|X|$ $%
\geq 2^{\aleph _{0}}$.

\noindent Further, notice that there is no bonding morphism between any pair
of terms having indices $\lambda \neq \lambda ^{\prime }$ in the same $%
\Lambda _{\text{\b{0}}}^{(n)}$, while every $p_{\lambda \lambda ^{\prime }}$%
, $\lambda <\lambda ^{\prime }$, is an epimorphism. but not an monomorphism
(i.e., not an isomorphism). The quite analogous partition

$|M_{\text{\b{0}}}|$ $=|Y^{\ast }|$ $=\dim Y^{\ast }=\dim Y=|Y|$ $\geq
2^{\aleph _{0}}$

\noindent of $M_{\text{\b{0}}}$ and the properties of $\boldsymbol{Y}_{\text{%
\b{0}}}$ hold as well. We shall prove that $|\Lambda _{\text{\b{0}}}|$ $=|M_{%
\text{\b{0}}}|$. Firstly, let us pass to the isomorphic cofinite inverse
systems ([11], Theorem I.1.2)

$\boldsymbol{X}_{\text{\b{0}}}^{\prime }=(X_{\bar{\lambda}}^{\prime },p_{%
\bar{\lambda}\lambda ^{\prime }}^{\prime },\bar{\Lambda}_{\text{\b{0}}%
})\cong \boldsymbol{X}_{\text{\b{0}}}$,

$\boldsymbol{Y}_{\text{\b{0}}}^{\prime }=(Y_{\bar{\mu}}^{\prime },q_{\bar{\mu%
}\bar{\mu}^{\prime }}^{\prime },\bar{M}_{\text{\b{0}}})\cong \boldsymbol{Y}_{%
\text{\b{0}}}$

\noindent (made of the same \textquotedblleft term-bond
material)\textquotedblright ) with $|\bar{\Lambda}_{\text{\b{0}}}|$ $\leq
|\Lambda _{\text{\b{0}}}|$ and $|\bar{M}_{\text{\b{0}}}|$ $\leq |M_{\text{\b{%
0}}}|$ (actually, the both \textquotedblleft $\leq $\textquotedblright\ are
\textquotedblleft $=$\textquotedblright ). By this passage (the construction
called \textquotedblleft Marde\v{s}i\'{c} trick\textquotedblright ), $\bar{%
\Lambda}_{\text{\b{0}}}$ is the disjoint union of $\bar{\Lambda}_{\text{\b{0}%
}}^{(n)}$, $n\in \{0\}\cup \mathbb{N}$, where $\bar{\Lambda}_{\text{\b{0}}%
}^{(0)}$ is the singleton containing the first (minimal) element, while, for
each $n\in \mathbb{N}$,

$\bar{\Lambda}_{\text{\b{0}}}^{(n)}=\{\bar{\lambda}\in \bar{\Lambda}_{\text{%
\b{0}}}\mid |\bar{\lambda}|$ $=n$ $\}$

\noindent ($|\lambda |$ denotes the cardinal of the set of all predecessors $%
\lambda _{j}<\lambda $) and

\TEXTsymbol{\vert}$\bar{\Lambda}_{\text{\b{0}}}^{(1)}|$ $=\cdots =|\bar{%
\Lambda}_{\text{\b{0}}}^{(n)}|$ $=\cdots =|\bar{\Lambda}_{\text{\b{0}}}|$.

\noindent Further, there is no bonding morphism between any pair of terms
having indices $\bar{\lambda}\neq \bar{\lambda}^{\prime }$ in the same $\bar{%
\Lambda}_{\text{\b{0}}}^{(n)}$, while every $p_{\bar{\lambda}\bar{\lambda}%
^{\prime }}$, $\bar{\lambda}<\bar{\lambda}^{\prime }$ (being $p_{\max \bar{%
\lambda}\max \bar{\lambda}^{\prime }}$), is an epimorphism., and an
monomorphism (i.e., an isomorphism) if and only if it is the identity, that
occurs when $\bar{\lambda}\subset \bar{\lambda}^{\prime }$ and $\max \bar{%
\lambda}=\max \bar{\lambda}^{\prime }$ only. The quite analogous partition
and properties hold for $\bar{M}_{\text{\b{0}}}$ and $\boldsymbol{Y}_{\text{%
\b{0}}}^{\prime }$. Therefore, we have to prove that $|\bar{\Lambda}_{\text{%
\b{0}}}|$ $=|\bar{M}_{\text{\b{0}}}|$. Since $Sh_{\text{\b{0}}}(X)=Sh_{\text{%
\b{0}}}(Y)$, there exist isomorphisms

$\boldsymbol{f}:\boldsymbol{X}_{\text{\b{0}}}^{\prime }\rightarrow 
\boldsymbol{Y}_{\text{\b{0}}}^{\prime }$, $\boldsymbol{g}=\boldsymbol{f}%
^{-1}:\boldsymbol{Y}_{\text{\b{0}}}^{\prime }\rightarrow \boldsymbol{X}_{%
\text{\b{0}}}^{\prime }$

\noindent of $pro$-$(\mathcal{B}_{F})_{\text{\b{0}}}\subseteq pro$-$(%
\mathcal{N}_{F})_{\text{\b{0}}}$. By [11], Lemmata I.1.2 and Remark I.1.8,
there exist special (the appropriate square diagrams commute)
representatives $(\phi ,f_{\mu }),$ $(\psi ,g_{\lambda })$ in $inv$-$(%
\mathcal{B}_{F})_{\text{\b{0}}}$ of $\boldsymbol{f}$, $\boldsymbol{f}^{-1}$
respectively. Then the subsystem

$\boldsymbol{X}_{0}^{\prime \prime }=(X_{\psi (\bar{\mu})}^{\prime },p_{\psi
(\bar{\mu})\psi (\bar{\mu}^{\prime })}^{\prime },\bar{M}_{\text{\b{0}}})$

\noindent of $\boldsymbol{X}_{\text{\b{0}}}^{\prime }$ is isomorphic to $%
\boldsymbol{X}_{\text{\b{0}}}^{\prime }$ in $pro$-$(\mathcal{B}_{F})_{\text{%
\b{0}}}$. It implies, by the mentioned properties of the terms and bonds of $%
\boldsymbol{X}_{\text{\b{0}}}^{\prime }$, that the index function $\psi :%
\bar{M}_{\text{\b{0}}}\rightarrow \bar{\Lambda}_{\text{\b{0}}}$ must be
cofinal, i.e.,

$(\forall \bar{\lambda}\in \bar{\Lambda}_{\text{\b{0}}})(\exists \bar{\mu}%
\in \bar{M}_{\text{\b{0}}})$ $\psi (\bar{\mu})\geq \bar{\lambda}$.

\noindent Since $\bar{\Lambda}_{\text{\b{0}}}$ is cofinite, this readily
implies that $|\bar{\Lambda}_{\text{\b{0}}}|$ $\leq \aleph _{0}\cdot |\bar{M}%
_{\text{\b{0}}}|$ $=|\bar{M}_{\text{\b{0}}}|$. One can establish, in the
same way, that $|\bar{M}_{\text{\b{0}}}|$ $\leq |\bar{\Lambda}_{\text{\b{0}}%
}|$ holds, and the conclusion folows. Consequently, $|\Lambda _{\text{\b{0}}%
}|$ $=|M_{\text{\b{0}}}|$, and therefore,

$\dim X^{\ast }=\dim Y^{\ast }=\dim Y$.

\noindent Then, by Theorem 2 (i) of [17], $Sh_{\text{\b{0}}}(X^{\ast })=Sh_{%
\text{\b{0}}}(Y)$ holds. We may now apply the same proof (to $X^{\ast }$ and 
$Y$) and conclude that $\dim X^{\ast \ast }=\dim Y=\dim X^{\ast }$.
Therefore, $X^{\ast }$ is $\dim ^{\ast }$-stable, whenever $\dim X\geq
2^{\aleph _{0}}$.

\noindent Assume now that $\dim X=\aleph _{0}$. Then $X\cong (F_{0}^{\mathbb{%
N}},\left\Vert \cdot \right\Vert )$, while $X^{\ast }\cong (F^{\mathbb{N}%
},\left\Vert \cdot \right\Vert ^{\ast })$, and $\dim X<\dim X^{\ast
}=2^{\aleph _{0}}$. Let $Y$ be a Hilbert space such that $\dim Y=2^{\aleph
_{0}}$. Since, by Theorem 2 (i) of [17], $Sh_{\text{\b{0}}}(X^{\ast })=Sh_{%
\text{\b{0}}}(Y)$, and since $Y$ is $\dim ^{\ast }$-stable, the first part
of the proof assures that $X^{\ast }$ is $\dim ^{\ast }$-stable, which
completes the proof.
\end{proof}

\begin{theorem}
\label{T4} Every normed vectorial space $X$ having (algebraic) $\dim X\neq
\aleph _{0}$ is $\dim ^{\ast }$-stable. Especially, every Banach space is $%
\dim ^{\ast }$-stable.
\end{theorem}

\begin{proof}
Firstly observe that in the special case of a $\dim ^{\ast }$-stable $Y=X$,
Corollary 3 and Lemma 4 imply that $Y^{\ast }$ is $\dim ^{\ast }$-stable,
i.e., $\dim Y^{\ast \ast }=\dim Y^{\ast }=\dim Y$. Then, by induction and
combining Lemma 4 with Theorem 2 (i) of [17], it follows that every iterated
normed dual of $Y$ is $\dim ^{\ast }$-stable, i.e., $\dim D^{n}(Y)=\dim Y$, $%
n\in \mathbb{N}$. Further, in the special case of a $\dim ^{\ast }$-stable $%
Y $ and $\dim X=\dim Y$, Lemma 4 and Theorem 2 (i) of [17] imply that $\dim
X^{\ast \ast }=\dim X^{\ast }=\dim Y=\dim X$. Hence, $X$ is $\dim ^{\ast }$%
-stable as well, and consequently, so are its all iterated normed duals $%
D^{n}(X)$. In this way we have proven that the $\dim ^{\ast }$-stability is
an invariant of the algebraic dimension, i.e., if $\dim X=\dim Y$ and $Y$ is 
$\dim ^{\ast }$-stable, then so is $X$, as well as, that the functor $D^{n}$
preserves $\dim ^{\ast }$-stability.

\noindent Clearly, the statement is not trivial in the infinite-dimensional
case only. Let $\dim X=\infty \neq \aleph _{0}$, implying that $\dim X\geq
2^{\aleph _{0}}$ ($GCH$ accepted), Let us choose a Hilbert space $Y$ (over
the same $F$) such that $\dim Y=\dim X$. Such a $Y$ exists, for instance, by
means of the usual construction. More precisely, let $J$ be an index set of
cardinality $|J|$ $=\dim X$, and let the set

$F^{J}=\{y\equiv (y_{j})\mid y:J\rightarrow F\}$

\noindent be endowed with the usual vectorial (algebraic) structure (over $F$%
). Consider its subspace

$F_{2}^{J}=\{y\in F^{J}\mid \sum_{j\in J}|y_{j}|$ $^{2}<\infty
\}\trianglelefteq F^{J}$.

\noindent Then $Y=(F_{2}^{J},\left\Vert \cdot \right\Vert _{2})$, where

$\left\Vert y\right\Vert _{2}=(\sum_{j\in J}|y_{j}|^{2})^{1/2}$,

\noindent is a Hilbert space. Moreover, since $|J|$ $\geq 2^{\aleph _{0}}$
and every $y\in F_{2}^{J}$ contains at most $\aleph _{0}$ non-zero
coordinates, one readily verifies that

$|F_{2}^{J}|$ $<|F|^{|J|}=2^{|J|}=|F^{J}|$.

\noindent Therefore, by Lemma 2 (iii) of [14]),

$\dim X\leq \dim F_{2}^{J}<\dim F^{J}=2^{|J|}=2^{\dim X}$.

\noindent It follows, by $GCH$, $\dim F_{2}^{J}=\dim X$, and thus, $\dim
Y=\dim X$. Since every Hilbert space is $\dim ^{\ast }$-stable, so is $X$.
\end{proof}

Observe that the assumption $\dim X\neq \aleph _{0}$ is essential because
of, for instance, $Sh_{\text{\b{0}}}((F_{0}^{\mathbb{N}},\left\Vert \cdot
\right\Vert _{p})^{\ast })=Sh_{\text{\b{0}}}(l_{p})$, $p>1$, while

$\dim (F_{0}^{\mathbb{N}},\left\Vert \cdot \right\Vert _{p})=\aleph
_{0}<2^{\aleph _{0}}=\dim l_{p^{\prime }}^{\ast }=\dim (F_{0}^{\mathbb{N}%
},\left\Vert \cdot \right\Vert _{p})^{\ast }$,

\noindent whenever $p^{-1}+(p^{\prime })^{-1}=1$.

An immediate consequence of Theorem 4 (and its proof) is the following fact.

\begin{corollary}
\label{C4}For every $X\in Ob(\mathcal{N}_{F})$,the following properties are
equivalent:

\noindent (i) $X$ is $\dim ^{\ast }$-stable;

\noindent (ii) $(\forall k\in \{0\}\cup \mathbb{N})$ $D^{k}(X)$ is $\dim
^{\ast }$-stable$,$ i.e., $\dim D^{k+1}(X)=\dim D^{k}(X)$.
\end{corollary}

We can now improve Corollaries 1 and 2 as well as Theorem 2 of [17], and
completely solve the \emph{finite} quotient shape classification of normed
vectorial spaces ($GCH$ assumed) as follows.

\begin{theorem}
\label{T5}For every $X\in Ob(\mathcal{N}_{F})$,

$\dim X^{\ast }=2^{\dim X}>\dim X\Leftrightarrow \dim X=\aleph _{0}$.

\noindent Equivalently,

$\dim X^{\ast }=\dim X\Leftrightarrow \dim X\neq \aleph _{0}$.

\noindent Therefore, for every $n\in \mathbb{N}$, $D^{n}(X)$ is $\dim ^{\ast
}$-stable, $\dim D^{n}(X)=\dim X^{\ast }$ and $Sh_{\text{\b{0}}%
}(D^{n}(X))=Sh_{\text{\b{0}}}(X)$. Consequently, given a pair $X,Y\in Ob(%
\mathcal{N}_{F})$, then

\noindent (i) $Sh_{\text{\b{0}}}(X)=Sh_{\text{\b{0}}}(Y)\Leftrightarrow
\left\{ 
\begin{array}{c}
\dim X=\dim Y\notin \{\aleph _{0},2^{\aleph _{0}}\} \\ 
\text{or} \\ 
\dim X,\dim Y\in \{\aleph _{0},2^{\aleph _{0}}\}%
\end{array}%
\right. $.

\noindent (ii) If $X,Y\in Ob(\mathcal{B}_{F})$, then

$(\dim X=\dim Y)\Leftrightarrow (Sh_{\text{\b{0}}}(X)=Sh_{\text{\b{0}}%
}(Y))\Leftrightarrow (Sh_{\aleph _{0}}(X)=Sh_{\aleph _{0}}(Y))$.

\noindent (iii) If $X$ and $Y$ are $\dim ^{\ast }$-stable and if there
exists a closed embedding $e:X\rightarrow Y$ such that $\dim (Y/e[X])<\dim
Y=\kappa $ ($\geq \aleph _{0})$, then

$(\dim X=\dim Y)\Leftrightarrow (Sh_{\kappa ^{-}}(X)=Sh_{\kappa ^{-}}(Y))$.
\end{theorem}

\begin{proof}
Concerning the first part and statement (i), i.e., the finite quotient shape
classification, we only need to verify that $Sh_{\text{\b{0}}}(X^{\ast
})=Sh_{\text{\b{0}}}(X)$ in the case $\dim X=\aleph _{0}$ as well. Indeed,
in that case, $X\cong (F_{0}^{\mathbb{N}},\left\Vert \cdot \right\Vert )$,
and $\dim X^{\ast }=2^{\aleph _{0}}$. By Corollary 4, $\dim D^{n}(X)=\dim
X^{\ast }$, for every $n\in \mathbb{N}$. Further, by Lemma 2 (iii) of [17],
the Banach completion $Cl(X)\subseteq X^{\ast \ast }$ rises (algebraic)
dimension$,$ i.e., $\dim Cl)X=2^{\aleph _{0}}$. Thus, $\dim Cl(X)=\dim
X^{\ast }$. Since, by Theorem 3 (i) of [16], $Sh_{\text{\b{0}}}(Cl(X))=Sh_{%
\text{\b{0}}}(X)$ holds, snd, by Theorem 2 (i) of [17], $Sh_{\text{\b{0}}%
}(Cl(X))=Sh_{\text{\b{0}}}(X^{\ast })$ holds, the conclusion follows.
Statement (ii) follows by (i) because there is no countably
infinite-dimensional Banach space. Statement (iii) generalizes Theorem 2
(iii) of [17] (for the bidual-like normed spacess) to the $\dim ^{\ast }$%
-stable normed spaces. In the proof of that theorem (especially, that of
Theorem 2 (ii)), Lemma 4 (ii) of [17] was used. However, its role in that
proof is the same as that of the $\dim ^{\ast }$-stability of $X$ and $Y$.
The conclusion follows.
\end{proof}

\begin{remark}
\label{R3}Although the second dual space $D^{2}(X)\equiv X^{\ast \ast }$ is
large enough to contain $X\subseteq Cl(X)\subseteq X^{\ast \ast }$ (as the
isometrically embedded subspaces), this enlargement does not rise dimension
neither cardinality because

$|X^{\ast \ast }|$ $=\dim X^{\ast \ast }=\dim Cl(X)=\dim X=|X|$,

\noindent whenever $\dim X>\aleph _{0}$. Therefore, in an
infinite-dimensional $\kappa ^{-}$-expansion of $X$, the codimension. of $%
Cl(X)$, i.e., $\dim (X^{\ast \ast }/Cl(X))$, plays the most important role.
\end{remark}

Observe that the cardinal arithmetic yield the following interesting
consequences.

\begin{corollary}
\label{C5}Let $X,Y\in Ob(\mathcal{N}_{F})$ such that $\dim X\geq \aleph _{0}$%
, $0<\dim Y<\dim X$ and $\dim Y\neq \aleph _{0}$, and let $L(X,Y)$ be the
(normed) space of all continuous linear functions of $X$ to $Y$. Then,

\noindent (i) $|L(X,Y)|$ $=\dim L(X,Y)=\dim D^{n}(X)=2^{\aleph _{0}}$, $n\in 
\mathbb{N}$, whenever $\dim X=\aleph _{0}$;

\noindent (ii) $|L(X,Y)|$ $=\dim L(X,Y)=\dim D^{n}(X)=\dim X=|X|$, $n\in 
\mathbb{N}$, whenever $\dim X>\aleph _{0}$.
\end{corollary}

\begin{proof}
For statement (i), notice that $Y\cong F^{k}$ for some $k\in \mathbb{N}$.
Then, since $\dim X\geq \aleph _{0}$, $L(X,Y)$ is an infinite-dimensional
Banach space, and thus, dim$L(X,Y)\geq 2^{\aleph _{0}}$. It follows, by
Lemma 3. 2. (iv) of [16], that dim$L(X,Y)=|L(X,Y)|$. Since $|L(X,Y)|$ $%
=|L(Cl_{X^{\ast \ast }}(X),Y)|$, the proof of Lemma 4 (the partition of the
index set of the canonical $(\mathcal{N}_{F})_{\text{\b{0}}}$-expansion of
\textquotedblleft $X$\textquotedblright\ $=Cl_{X^{\ast \ast }}(X)$) shows
that $|L(Cl_{X^{\ast \ast }}(X),Y)|$ $=|(Cl_{X^{\ast \ast }}(X))^{\ast }|$.
By Lemma 3 2. (iv) of [16] again, $|(Cl_{X^{\ast \ast }}(X))^{\ast }|$ $%
=\dim (Cl_{X^{\ast \ast }}(X))^{\ast }$. Finally, by Theorem 5,

$\dim (Cl_{X^{\ast \ast }}(X))^{\ast }=\dim Cl_{X^{\ast \ast }}(X)=2^{\aleph
_{0}}$,

\noindent and the conclusion follows. For statement (ii), let $\dim X=\kappa
\geq 2^{\aleph _{0}}$ and $\dim Y=\kappa ^{\prime }<\kappa $, $\kappa
^{\prime }\neq \aleph _{0}$. Then, by Theorem 5 and Lemma 3. 2. (iv) of [16],

$\kappa \leq \dim L(X,Y)=\dim L(X,Y)^{\ast }=\dim L(Y^{\ast },X^{\ast
})=|L(Y^{\ast },X^{\ast })|$ $\leq |X^{\ast }|^{|Y^{\ast
}|}=|X|^{|Y|}=\kappa ^{\kappa ^{\prime }}=\kappa $,

\noindent and the conclusion follows.
\end{proof}

Concerning the extensions of morphisms, the following extension type theorem
is an immediate consequence of Theorem 1 (i.e., Corollary 1).

\begin{theorem}
\label{T6}Let $X,Y\in Ob(\mathcal{N}_{F})$ and let $f_{n}:D^{n}(X)%
\rightarrow Y$, $n\in \mathbb{N}$, be a continuous linear function. Then,
for every $k\in \{0\}\cup \mathbb{N}$, $f_{n}$ admits a continuous linear
extension $f_{n,k}:D^{n+2k}(X)\rightarrow Y$. If, in addition $Y$ is a
Banach space, then $\left\Vert f_{n,k}\right\Vert =\left\Vert
f_{n,k}\right\Vert $.
\end{theorem}

\begin{proof}
Let an $X,Y\in OB\mathcal{N}_{F}$, an $n\in \mathbb{N}$, and an $f_{n}\in 
\mathcal{N}_{F}(D^{n}(X),Y)=\mathcal{N}_{F}(D^{n}(X),Y)$ be given. If $k=0$,
there is nothing to prove. Let $k>0$. By Theorem 1 (and its proof), the
canonical embedding $j_{n}$ is a section having $D(j_{n-1})$ for an
appropriate retraction$,$ $D(j_{n-1})j_{n}=1_{D^{n}(X)}$. Then

$f_{n,1}=f_{n}D(j_{n-1}):D^{n+2}(X)\rightarrow Y$

\noindent is a desired extension when $k=1$. Assume that $Y$ is a Banach
space. By the proof of Corollary 1, for every $n\in \mathbb{N}$, the morphism

$p_{n+2}\equiv D^{n-1}(j_{1}D(j_{0})):D^{n+2}(X)\rightarrow D^{n+2}(X)$

\noindent is a continuous linear projection onto the retract $R(j_{n})\equiv
D^{n}(X)$ of $D^{n+2}(X)$. One readily sees that $\left\Vert
j_{1}D(j_{0})\right\Vert =1$, and thus,

$\left\Vert p_{n+2}\right\Vert =\left\Vert D^{n-1}(j_{1}D(j_{0}))\right\Vert
=\left\Vert j_{1}D(j_{0})\right\Vert =1$.

\noindent This implies the existence of a desired extension $%
f_{n,1}:D^{n+2}(X)\rightarrow Y$ (see also Proposition 6.6.18. of [6]).
Thus, in the case $k=1$, the statements are proven. The rest follows by
induction on $k$.
\end{proof}

The following extension type theorem is an improvement of Theorem 4 of [17]
(see also Lemma 3 (ii) of [17]) .

\begin{theorem}
\label{T7}Let $X$ be a normed space, let $Z\trianglelefteq X$ be a subspace
such that $\dim Cl(Z)=\dim X$ and $\dim (X/Cl(Z))<\dim X$, and let $Y$ be a
Banach space (over the same field) having $\dim Y<\dim X$. Then every
continuous linear function $f:Z\rightarrow Y$ admits a continuous linear
norm-preserving extension $\bar{f}:X\rightarrow Y$.
\end{theorem}

\begin{proof}
Clearly, only the infinite-dimensional case asks for a proof. Let $\dim
X=\kappa \geq \aleph _{0}$, and let $Z\trianglelefteq X$, $Y$ and $%
f:Z\rightarrow Y$ be given according to the assumptions. By Lemma 1 of [17],
there exists a (unique) continuos linear extension $f^{\prime
}:Cl(Z)\rightarrow Y$ of $f$, and $\left\Vert f^{\prime }\right\Vert
=\left\Vert f\right\Vert $. Firstly, consider the case $\dim X=\aleph _{0}$.
Then, $X\cong (F_{0}^{\mathbb{N}},\left\Vert \cdot \right\Vert )$ and $\dim
Cl(Z)=\aleph _{0}$, while $X/Cl(Z)$ and $Y$ are finite-dimensional by
assumption. By Theorem 5,

$\dim X=\dim Cl(Z)<\dim X^{\ast \ast }=\dim Cl(Z)^{\ast \ast }=2^{\aleph
_{0}}$,

\noindent and, by Theorem 2 (ii),

$X^{\ast \ast }/Cl(Z)^{\ast \ast }\cong (X/Cl(Z))^{\ast \ast }$

\noindent (that is, in this case, even isomorphic to $X/Cl(Z)$). Denote by

$Cl_{X^{\ast \ast }}(Cl(Z))\subseteq Cl_{X^{\ast \ast }}(X)$ $\subseteq
X^{\ast \ast }$

\noindent the closures (the Banach completions) of $Cl(Z)\subseteq X$ in $%
X^{\ast \ast }$, and by $f^{\prime \prime }:Cl_{X^{\ast \ast
}}(Cl(Z))\rightarrow Y$ the continuous (unique, linear) extension of $%
f^{\prime }$. Since the canonical embedding into the second dual space is an
isometry, $\left\Vert f^{\prime \prime }\right\Vert =\left\Vert f^{\prime
}\right\Vert $ holds. One readily sees (see also Lemma 2 (iii) of [17]) that

$\dim (Cl_{X^{\ast \ast }}(Cl(Z)))=\dim (Cl_{X^{\ast \ast }}(X))$.

\noindent and, since $\dim (X/Cl(Z))<\aleph _{0}$, that

dim($Cl_{X^{\ast \ast }}(X)/Cl_{X^{\ast \ast }}(Cl(Z)))<\aleph _{0}$.

\noindent By applying Theorem 5 (iii) to $Cl_{X^{\ast \ast
}}(Cl(Z))\hookrightarrow Cl_{X^{\ast \ast }}(X)$ (they are Banach spaces,
hence $\dim ^{\ast }$-stable) or Theorem 5 (ii) only, it follows that

$Sh_{\text{\b{0}}}(Cl_{X^{\ast \ast }}(Cl(Z)))=Sh_{\text{\b{0}}}(Cl_{X^{\ast
\ast }}(X))$.

\noindent By means of that fact, we shall prove that $f^{\prime \prime }$
admits a continuous linear norm-preserving extension to $Cl_{X^{\ast \ast
}}(X)$. Then a desired extension $\bar{f}:X\rightarrow Y$ of $f$ may be the
restriction to $X$ of that extension. Instead of proving this separately (in
the special countably dimensional case), let us consider the general
uncountably dimensional case, i.e., $\dim X=\kappa >\aleph _{0}$, and thus, $%
\dim Cl(Z)=\kappa $, $\dim (X/Cl(Z))<\kappa $ and $\dim Y<\kappa $. As in
the countably dimensional case before, $f:Z\rightarrow Y$ admits a
continuous linear norm-preserving extension $f^{\prime \prime }:Cl_{X^{\ast
\ast }}(Z)\rightarrow Y$. By Theorem 5,

$\dim X^{\ast \ast }=\dim X=\dim Cl(Z)=\dim Cl(Z)^{\ast \ast }=\kappa $,

\noindent and, by Theorem 2 (ii) and Lemma 2 (iii) of [17],

dim($Cl_{X^{\ast \ast }}(X)/Cl_{X^{\ast \ast }}(Cl(Z)))<\kappa $.

\noindent By Theorem 5 (iii),

$Sh_{\kappa ^{-}}(Cl_{X^{\ast \ast }}(Cl(Z)))=Sh_{\kappa ^{-}}(Cl_{X^{\ast
\ast }}(X))$.

\noindent So, concerning the quotient shapes result, the uncountably
dimensional case covers the countably dimensional case as well. According to
Theorem 3 of [17], it suffices to prove that, by the inclusion $%
i:Cl_{X^{\ast \ast }}(Cl(Z))\hookrightarrow Cl_{X^{\ast \ast }}(X)$, induced
quotient shape morphism

$S_{\kappa ^{-}}(i):Cl_{X^{\ast \ast }}(Cl(Z))\hookrightarrow Cl_{X^{\ast
\ast }}(X)$

\noindent is an isomorphism of $Sh_{\kappa ^{-}}(\mathcal{B}_{F})$. To prove
this, we may use the appropriate part of the proof of [17], Theorem 4.
Indeed, that proof is based on Theorems 2 and 3 of [17]$,$ and the proof of
[17], Theorem 2, depends on Lemmata 3 and 4 of [17]. Especially, Lemma 4
(ii) of [17] dictated the restriction to the separable or bidual-like normed
spaces in order to keep control over the index sets in the considered
quotient expansions. However, that control is nothing else but the $\dim
^{\ast }$-stability of the considered normed spaces. The conclusion follows.
\end{proof}

\begin{corollary}
\label{C6}Let $Z$ be a dense subspace of an $X\in Ob(\mathcal{N}_{F})$ and
let $Y\in Ob(\mathcal{B}_{F})$ having $\dim Y<\dim \bar{X}$, where $\bar{X}$
denotes the (canonical) Banach completion $Cl_{D^{2}(X)}(j_{0}[X])$ of $X$.
If $\dim (D^{2}(X)/\bar{X})<\dim D^{2}(X)$, then, for every $n\in \mathbb{N}$%
, every continuous linear function $f:Z\rightarrow Y$ admits a continuous
linear norm-preserving extension $f_{n-1}:D^{2n}(X)\rightarrow Y$.
\end{corollary}

\begin{proof}
Firstly, if $\dim X<\aleph _{0}$, then $Z=X\cong D^{n}(X)$, and there is
nothing to prove. So, let $\dim X=\infty $. Then, since $\bar{X}$ is a
Banach space, $\dim \bar{X}\geq 2^{\aleph _{0}}$ holds. Notice that, by
Theorems 3 and 4, $\dim \bar{X}=2^{\aleph _{0}}>\dim X$ if and only if $\dim
X=\aleph _{0}$. Further, since $Y$ is a Banach space, either $\dim Y<\aleph
_{0}$ or $\dim Y\geq 2^{\aleph _{0}}$ holds. Therefore, in any case, $\dim
Y<\dim \bar{X}$ implies $\dim Y<\dim X$. Let $n=1,$ and consider the
corresponding diagram (in $\mathcal{N}_{F}$), i.e.,

$%
\begin{array}{c}
Z\hookrightarrow \\ 
f\searrow \\ 
\end{array}%
\begin{array}{ccc}
Cl(Z)=X\overset{\cong }{\rightarrow }j_{0}[X]\hookrightarrow & \bar{X}%
\hookrightarrow & D^{2}(X) \\ 
\downarrow f^{\prime } & \swarrow f^{\prime \prime }\quad \swarrow f_{0} & 
\\ 
Y &  & 
\end{array}%
$

\noindent in which suitable extensions $f^{\prime }$ and $f^{\prime \prime }$%
, $f^{\prime }|Z=f$ and $f^{\prime \prime }|\bar{X}=f^{\prime }$, exist by
Lemma 1 of [17], and then, a desired extension $f_{0}$ exists by Theorem 7.
Now, for $n>1,$ one applies Theorem 6.
\end{proof}

\begin{remark}
\label{R4}Since the \emph{finite} quotient shape classification of normed
spaces reduces to the relations of their algebraic dimensions established by
Theorem 5 (similarly to that of all vectorial spaces - [14] Theorem 3, (i) $%
\Leftrightarrow $ (v)), in order to get a deeper insight in this matter, one
has to consider the higher dimensional quotient shape classifications. It
asks, however, for an insight into the \textquotedblleft dark
area\textquotedblright\ of the structure of the set of all \emph{closed}
subspaces of an infinite ($\kappa $-) dimensional normed space all having
the (algebraic) dimension of the space and \emph{infinite codimension.} up
to a given infinite cardinal less than $\kappa $. Notice that the essential
fact (for the finite quotient shape classification) that all norms on a
finite-dimensional space are equivalent does not hold any more in the
infinite-dimensional case. Further, the specific properties of the bonding
morphisms between the terms having indices in a $\Lambda _{\text{\b{0}}%
}^{(n)}$ (used in the proof of Lemma 3) do not hold in the case of $\Lambda
_{\kappa ^{\prime }}^{(\kappa )}$ whenever $\dim X=\kappa ^{\prime }>\kappa
\geq \aleph _{0}$.\emph{\ \bigskip }
\end{remark}

\begin{center}
\textbf{References\smallskip }
\end{center}

\noindent \lbrack 1] K. Borsuk, \textit{Concerning homotopy properties of
compacta}, Fund. Math. \textbf{62} (1968), 223-254.

\noindent \lbrack 2] K. Borsuk, \textit{Theory of Shape}, Monografie
Matematyczne \textbf{59}, Polish Scientific Publishers, Warszawa, 1975.

\noindent \lbrack 3] J.-M. Cordier and T. Porter, \textit{Shape Theory:
Categorical Methods of Approximation}, Ellis Horwood Ltd., Chichester, 1989.
(Dover edition, 2008.)

\noindent \lbrack 4] J. Dugundji, \textit{Topology}, Allyn and Bacon, Inc.
Boston, 1978.

\noindent \lbrack 5] Dydak and J. Segal, \textit{Shape theory: An
introduction}, Lecture Notes in Math. \textbf{688}, Springer-Verlag, Berlin,
1978.

\noindent \lbrack 6] J. M. Erdman, \textit{Functional Analysis and Operator
Algebras - An Introduction}$,$ Version October 4, 2015, Portland State
University (licensed PDF)

\noindent \lbrack 7] H. Herlich and G. E. Strecker, \textit{Category Theory,
An Introduction}, Allyn and Bacon Inc., Boston, 1973.

\noindent \lbrack 8] N. Kocei\'{c} Bilan and N. Ugle\v{s}i\'{c}, \textit{The
coarse shape}, Glasnik. Mat. \textbf{42}(\textbf{62}) (2007), 145-187.

\noindent \lbrack 9] E. Kreyszig, \textit{Introductory Functional Analysis
with Applications}, John Wiley \& Sons, New York, 1989.

\noindent \lbrack 10] S. Kurepa, \textit{Funkcionalna analiza : elementi
teorije operatora}, \v{S}kolska knjiga, Zagreb, 1990.

\noindent \lbrack 11] S. Marde\v{s}i\'{c} and J. Segal, \textit{Shape Theory}%
, North-Holland, Amsterdam, 1982.

\noindent \lbrack 12] W. Rudin, \textit{Functional Analysis, Second Edition}%
, McGraw-Hill, Inc., New York, 1991.

\noindent \lbrack 13] N. Ugle\v{s}i\'{c}, \textit{The shapes in a concrete
category}, Glasnik. Mat. Ser. III \textbf{51}(\textbf{71}) (2016), 255-306.

\noindent \lbrack 14] N. Ugle\v{s}i\'{c}, \textit{On the quotient shape of
vectorial spaces}, Rad HAZU - Matemati\v{c}ke znanosti, Vol. \textbf{21} = 
\textbf{532} (2017), 179-203.

\noindent \lbrack 15] N. Ugle\v{s}i\'{c}, \textit{On the quotient shapes of
topological spaces}, Top. Appl. \textbf{239} (2018), 142-151.

\noindent \lbrack 16] N. Ugle\v{s}i\'{c}, \textit{The quotient shapes of }$%
l_{p}$\textit{\ and }$L_{p}$\textit{\ spaces}, Rad HAZU. Matemati\v{c}ke
znanosti, Vol. \textbf{27} = \textbf{536} (2018), 149-174.

\noindent \lbrack 17] N. Ugle\v{s}i\'{c}, \textit{The quotient shapes of
normed spaces and application}, submitted.

\noindent \lbrack 18] N. Ugle\v{s}i\'{c} and B. \v{C}ervar, \textit{The
concept of a weak shape type}, International J. of Pure and Applied Math. 
\textbf{39} (2007), 363-428.\vspace{0.3in}

\begin{center}
O DUALIMA NORMIRANIH PROSTORA I KVOCIENTNIM OBLICIMA\smallskip
\end{center}

Nekoliko svojstava normiranoga dualnog Hom-funktora $D$ $i$ njegovih
iteracija $D^{n}$ je ustanovljeno. Primjerice: $D$ preokre\'{c}e svako
kanonsko smje\v{s}tenje (u drugi dualni prostor) u retrakciju (tre\'{c}ega
dualnog prostora na onaj prvi); $D$ povisuje (algebrasku) dimenziju \emph{%
samo} prebrojivo bezkona\v{c}no-dimenzionalnim normiranim prostorima; $D$ ne
mienja kona\v{c}ni kvocientni oblikovni tip. Spomo\'{c}u toga je podpuno rie%
\v{s}ena razredba svih normiranih vektorskih prostora po \emph{kona\v{c}nomu}
kvocientnom tipu. Kao primjena, za posljedicu su izvedena dva pou\v{c}ka o
pro\v{s}irivanju neprekidnih linearnih funkcija.

\end{document}